%% file: smoothK.tex
\documentclass[a4paper,11pt,twoside]{amsart}

\usepackage{amssymb}
\usepackage{vmargin}
\usepackage{mathrsfs}
\usepackage[all]{xy}

\setmargins{32mm}{20mm}{14.6cm}{22cm}{1cm}{1cm}{1cm}{1cm}

\include{newcommands}

\include{newthms}

\begin{document}

\begin{abstract}
 
We show that real Deligne cohomology of a 
complex manifold $X$ arises locally 
as a 
topological vector 
space completion of the analytic 
Lie groupoid of holomorphic vector bundles. Thus Beilinson's regulator arises naturally as a comparison map between $K$-theory groups of different types.

\end{abstract}

\title{A $K$-theoretic interpretation of real Deligne cohomology}
\author{J.P.Pridham}
\thanks{This work was supported by  the Engineering and Physical Sciences Research Council [grant number EP/I004130/2].}


\maketitle
\section*{Introduction}

The main purpose this paper is to reinterpret Beilinson's regulator as a morphism between $K$-theory groups of different types, and thus to describe real Deligne cohomology as a new form of $K$-theory.

Our motivation comes from considering the $0$-dimensional case.  The $\R$-linearised algebraic $K$-group $K_1(\Cx)_{\R}=\Cx^{\by}\ten_{\Z}\R$ is far larger than the real Deligne cohomology group $\H^1_{\cD}(\Spec \Cx, \R(1))\cong\R$. However,  if we regard $\Cx^{\by}$ as a Lie group,  then its maximal $\R$-linear quotient  is $\Cx^{\by}/S^1 \cong \R$, with the quotient realised by  the Beilinson regulator. Proceeding further, the real Deligne cohomology complex of $\Cx$ can be regarded as the universal $\R$-linear completion of the monoid $\coprod_n B\GL_n(\Cx)$ in Lie groupoids.

To understand the differences between this theory, algebraic and topological $K$-theory, consider the following calculations of continuous group homomorphisms for real vector spaces $V$, where $\delta$ denotes the discrete topology:
\begin{align*}
  \Hom_{\R}(K_1(\Cx)_{\R},V)&\cong\Hom((\Cx^{\by})^{\delta},V)\cong \Hom_{\R}(\Cx^{\by}\ten_{\Z}\R,V)  \\
\Hom_{\R}(K_1^{top}(\Cx)_{\R},V)&\cong \Hom(\pi_0\Cx^{\by},V)\cong 
\Hom_{\R}(0,V)  \\
\Hom_{\R}(\H^1_{\cD}(\Spec \Cx, \R(1)),V)&\cong \Hom(\Cx^{\by},V)\cong 
\Hom_{\R}(\R,V),
\end{align*}
the  isomorphisms on the middle line following because  $\pi_1(B\Cx^{\by})=\pi_0\Cx^{\by}$. 
Thus the idea behind our topological $\R$-linear completion is to look at features of  manifolds and Lie groups which are not encoded by the underlying set of points or by the homotopy type, but by  spaces of smooth functions, or dually by the space of compactly supported distributions. 

Another way to formulate our characterisation above of the real Deligne cohomology complex as a completion  is to say that 
it is 
the universal complex $V$ of real topological vector spaces equipped with a smooth map $K(\Cx) \to V$, by which we mean a suitably coherent system of
maps
\[
 K(\C^{\infty}(Z, \Cx)) \to \C^{\infty}(Z, V)
\]
functorial in  Fr\'echet manifolds $Z$, where $\C^{\infty}(Z,A)$ denotes the ring of smooth $A$-valued functions on $Z$.

For a complex manifold $X$, we therefore consider
the presheaf
\[
 Z \mapsto K(\C^{\infty}(Z, \sO_X^{\an}))
\]
on a category of Fr\'echet manifolds $Z$, for the sheaf $\sO_X^{\an}$ of holomorphic functions on $X$,
and look at maps from this to presheaves
\[
 Z \mapsto \C^{\infty}(Z, \sV)
\]
 for hypersheaves  $\sV$ of complexes of real Banach spaces on  $X$. It turns out that the presheaf $K_{\Ban}(\sO_X^{\an})$ pro-representing these smooth functions in the homotopy category is, up to hypersheafification,  just the real Deligne complex 
\[
 \bigoplus_{p\ge 0} \R_{\cD,X}(p)^{2p-\bt}, 
\]
regarded as a complex of Fr\'echet spaces.

We then have
\[
K_{\Ban}(X):=  \oR\Gamma(X_{\an},K_{\Ban}(\sO_X^{\an})) \simeq \bigoplus_{p \ge 0} \oR\Gamma(X, \R_{\cD}(p)),
\]
and the natural morphisms $K(X) \to K_{\Ban}(\sO_X^{\an})$ induce a morphism $K(X) \to K_{\Ban}(X)$, which is just Beilinson's regulator.

Heuristically, the proof proceeds as follows, imitating many of the steps of Beilinson's comparison between  his and Borel's regulators, as expounded in \cite{BurgosGilReg}.  For technical reasons, we consider $1$-connective $K$-theory $K_{>0}$ instead of $K$.

To see that Beilinson's regulator induces an equivalence in a homotopy category of complexes of pro-Banach spaces, we need to show that it induces isomorphisms
\[
 \ext^i_{\dif}(\bigoplus_{p > 0}\R_{\cD}(p), \sV) \to \ext^i_{\dif}(K_{\Ban,>0}(\sO_X^{\an}),\sV)
\]
on (appropriately defined) differentiable $\Ext$  groups, for all  sheaves $\sV$ of Banach spaces on $X_{\an}$.
Since $K_{>0}(\C^{\infty}(Z,\sO_X^{\an}))$ is locally equivalent to $ B\C^{\infty}(Z,\GL(\sO_X^{\an}))^+$ for all Fr\'echet spaces $Z$,  the sheaf 
\[
 \sH^i_{\dif}(K_{>0}(\sO_X^{\an}),\sV) \subset \sH^i_{\dif}(B\GL(\sO_X^{\an}),\sV)
\]
 should then consist of primitive elements in differentiable cohomology.

The symmetric space $U_n \backslash \GL_n(\sO_{X,x}^{\an})$  is contractible at all points $x \in X$, so Hochschild and Mostow's calculation adapts to give 
\[
 \sH^i_{\dif}(B\GL_n(\sO_X^{\an}),\sV)\cong \sH^i(U_n \backslash \GL_n(\sO_X),\sV)^{\GL_n(\sO_X)}.
\]
 This is just relative Lie algebra cohomology $\sH^i(\gl_n\ten \sO_X^{\an}, \fu_n;\sV)$, and primitive elements as above then correspond to primitive elements of $\sH^i(\gl\ten \sO_X^{\an}, \fu;\sV) $. Thus the pro-Banach completion of $K_{>0}(\sO_X^{\an}) $ is locally equivalent to the complex $P\hat{E}_{\bt}( \gl\ten \sO_X^{\an}, \fu)$ of primitive elements for relative Lie algebra homology.          

We then appeal to Loday and Quillen's calculation, which adapts to show that primitive elements in Lie algebra homology are given by  cyclic homology. This in turn is equivalent to the filtered de Rham complex, so $P\hat{E}_{\bt}( \gl\ten \sO_X^{\an}) \simeq \bigoplus_{p>0}  \Omega^{\bt}_X[2p-1]/F^p$, and combined with the homology of $\fu$, this yields
\[
K_{\Ban,>0}(\sO_X^{\an})\simeq \bigoplus_{p> 0} \cone(\R(p) \to \Omega^{\bt}_X/F^p)^{2p-1-\bt}= \bigoplus_{p> 0} \R_{\cD,X}(p)^{2p-\bt}.
\]

The structure of the paper is as follows.

Section \ref{completionsn} introduces the notion of derived pro-Banach completion $\oL\ban$ of a presheaf on Fr\'echet manifolds, and establishes various foundational results concerning presheaves of pro-Banach complexes. The precise category of Fr\'echet manifolds considered is essentially unimportant --- it just has to have enough test objects to detect all smooth morphisms we encounter.  Proposition \ref{banmodelprop} constructs a model structure on the category of pro-Banach complexes, and  \S \ref{Kbansn} introduces, for Fr\'echet algebras $A$, the $K$-theory presheaf $\fK_{>0}(A)(Z):= K_{>0}(\C^{\infty}(Z,A))_{\R}$     and its pro-Banach completion $K_{\Ban,>0}(A)=\oL\ban \fK_{>0}(A)$. 

The main results of the paper all appear in \S \ref{Kequivsn}, which is concerned with calculating $K_{\Ban,>0}(\sO_X^{\an})$. 
Corollary  \ref{cohoderhamprop} relates differentiable cohomology of $\GL_n(\sO_X^{\an})$ to invariant forms on the symmetric space $U_n \backslash \GL_n(\sO_X^{\an})$, and 
Proposition \ref{banLiecor} applies this to give local strict quasi-isomorphisms between the continuous relative Lie algebra homology complex $\hat{E}_{\bt}^{\R}(\gl_n(\sO_X^{\an}), \fu_n) $ and the pro-Banach completion $\oL\ban \CC_{\bt}(B\GL_n(\sO_X^{\an}),\R)$ of the  complex of  real chains.

Combining these results with the plus construction, Proposition \ref{LieSymmKprop} gives 
$K_{\Ban, >0}(\sO_X^{\an})$ as a direct summand of $\hat{E}_{\bt}^{\R}(\gl(\sO_X^{\an}), \fu)$.
Theorem \ref{mainthm0} strengthens this to a local strict quasi-isomorphism between $K_{\Ban, >0}$ and the cyclic homology complex.
Theorem \ref{mainthm} then uses this to show that  $K_{\Ban,>0}$ and the real Deligne complex are locally strictly quasi-isomorphic. 

Since the functor $\C^{\infty}(Z,-)$ of smooth functions is pro-represented by the space of compactly supported distributions on $Z$, our derived pro-Banach completion functor $\oL\ban$ can be thought of as an enriched left Kan extension of the functor of compactly supported distributions. In particular, this means that
$K_{\Ban,>0}(A)$ can be regarded as the complex of compactly supported distributions on the presheaf $\fK_{>0}(A)$.
A related result has since been established in the sequel \cite{fdK}, which   looks at the non-connective $K$-theory presheaf $U \mapsto \bK(U \by Y)$ on the category of complex affine schemes, for any smooth proper complex scheme $Y$ (and derived and non-commutative generalisations), and shows that the space of compactly supported  distributions on this presheaf  is closely related to real Deligne cohomology.

I would like to heartily thank Ulrich Bunke for identifying many omissions and errors, and to thank the anonymous referee for many helpful suggestions.

\subsection*{Notation}

We work systematically with chain complexes and  homological grading conventions. In particular,  shifts of chain complexes are denoted  $V[n]_i:= V_{n+i}$. 
We always write $\cone(V)$ for the cone of the identity map on $V$.

\tableofcontents

\section{Pro-Banach completion}\label{completionsn}

\subsection{Pro-Banach complexes}

\begin{definition}
 Define $\Ch(\pro(\Ban))$ to be the category of chain complexes $\ldots \xra{d} V_{1} \xra{d} V_0 \xra{d} V_{-1} \xra{d} \ldots $ in the pro-category $\pro(\Ban)$ of the category $\Ban$ of real Banach spaces and continuous linear maps.
\end{definition}

We will silently make use of the embedding $\Ban \to \pro(\Ban)$ sending a Banach space to the constant inverse system of Banach spaces. We will also identify $\pro(\Ban)$ with the subcategory of $\Ch(\pro(\Ban))$ consisting of complexes concentrated in degree $0$. By \cite[Proposition 3.3.4]{prosmans}, the category of Fr\'echet spaces naturally embeds in $\pro(\Ban)$.

\begin{definition}\label{weakproban}
 Say that a morphism $A \to B$ in $\Ch(\pro(\Ban))$ is a 
strict quasi-isomorphism
if its cone $C$ is strictly acyclic, in the sense that the maps $d \co C_i/\z_iC \to \z_{i-1}C$ are isomorphisms in $\pro(\Ban)$.
\end{definition}

\begin{definition}\label{HHomBandef}
Given $U,V \in  \Ch(\pro(\Ban))$, define the chain complex $\HHom_{\Ban}(U,V)$ of real vector spaces by
\[
 \HHom(U,V)_n := \prod_i \Hom_{\pro(\Ban)}(U_i, V_{i+n}),
\]
with differential $df: = d_V \circ f + (-1)^i f \circ d_U$, following  the sign trick for double complexes. 
\end{definition}

\begin{lemma}\label{weakbanlemma}
A morphism $A \to B$ in $\Ch(\pro(\Ban))$ is a 
strict quasi-isomorphism
if and only if $\HHom_{\Ban}(B, \ell^{\infty}(I)) \to \HHom(A, \ell^{\infty}(I))$ is a quasi-isomorphism for all small sets $I$.
\end{lemma}
\begin{proof}
 This follows from  \cite[Proposition 3.3.3]{prosmans}.
\end{proof}

\begin{proposition}\label{banmodelprop}
There is a fibrantly cogenerated model structure on the category of chain complexes of pro-Banach spaces, with cogenerating fibrations $P=\{\cone(\ell^{\infty}(I))[n] \to \ell^{\infty}(I)[n-1]\}_{I,n}$ and  cogenerating trivial fibrations $Q=\{\cone(\ell^{\infty}(I))[n] \to 0\}_{I,n}$. Weak equivalences are strict quasi-isomorphisms, and  cofibrations are levelwise strict monomorphisms.
\end{proposition}
\begin{proof}
We verify the conditions of \cite{Hovey} Theorem 2.1.19.
\begin{enumerate}
 \item   The two-out-of-three property and closure under retracts follow immediately from Lemma \ref{weakbanlemma}.

\item[(2)--(3).] Since the codomains of  $P$ and $Q$ are finite complexes of Banach spaces, they are cosmall in $\Ch(\pro(\Ban))$.

\item[(4).] Every morphism in $Q$ is a pullback of a morphism in $P$, and hence a $P$-cocell, so every $Q$-cocell is a $P$-cocell and hence in $P$-fib. Every pullback of a morphism in $Q$ is a strict quasi-isomorphism, so every  $Q$-cocell is a strict quasi-isomorphism, by Lemma \ref{weakbanlemma}.

\item[(5--6).] 
Consider the classes $P$-proj and $Q$-proj of morphisms with the left lifting property with respect to $P$ and $Q$-respectively. 
Applying   \cite[Proposition 3.3.3]{prosmans}, we see that $Q$-proj consists of levelwise strict monomorphisms. Meanwhile,  $P$-proj consists of morphisms $f\co A \to B$ for which the maps
\[
 \HHom_{\Ban}(B, \ell^{\infty}(I)) \to \HHom_{\Ban}(A, \ell^{\infty}(I))
\]
are surjective quasi-isomorphisms for all $I$. Surjectivity is equivalent to the condition that $f$ lies in $Q$-proj, so by Lemma \ref{weakbanlemma}, $P$-proj consists of the strict quasi-isomorphisms in $Q$-proj.

\end{enumerate}
\end{proof}

\subsection{Completion}

As in 
\cite{KrieglMichor},
it makes sense to talk about smooth functions and smooth differential forms on Fr\'echet  manifolds. 

\begin{definition}
Write $\FrM$ for the category consisting of 
open subsets of separable Fr\'echet spaces, 
with morphisms given by smooth (i.e. $\C^{\infty}$) maps in the sense of \cite[Definition 3.11]{KrieglMichor}.
\end{definition}
Explicitly, a morphism $f \co M \to N$ is said to be smooth if $f \circ c \co \R \to N$ is smooth for all smooth maps $c \co \R \to M$. Note that by \cite[Theorem 4.11]{KrieglMichor}, the locally convex topology on a Fr\'echet space is the same as the $c^{\infty}$-topology, so every object of $\FrM$ is $c^{\infty}$-open in a locally convex space, and the notion of smooth maps given in \cite[Definition 3.11]{KrieglMichor} makes sense for objects of $\FrM$.  

That smooth morphisms are closed under composition follows from the  characterisation 
\[
 \C^{\infty}(M,N) = \Lim_{c \in \C^{\infty}(\R,M)}\C^{\infty}(\R,N)
\]
in \cite[Lemma to Theorem 3.12]{KrieglMichor}. This can also be taken as the definition of the topology on $ \C^{\infty}(M,N)$, where the topology on $\C^{\infty}(\R,N)$ is given by uniform convergence of each  derivative separately on compact subsets (\cite[Definition 3.6]{KrieglMichor}).

By \cite[Theorem 3.12]{KrieglMichor}, an exponential law holds: for any locally convex space $F$, and $M,N \in \FrM$, we have a canonical isomorphism 
\[
 \C^{\infty}(M \by N,F ) \cong \C^{\infty}(M, \C^{\infty}(N,F)). 
\]

\begin{remark}
Observe that $\FrM$ is equivalent to a small category. This follows because the set of isomorphism classes of separable Fr\'echet spaces is small, since each is determined by a countable system of seminorms on a vector space with countable basis.

Instead of $\FrM$, there are many small full subcategories $\C$ of Fr\'echet manifolds which we could use. The category must contain all finite sets, and
applications of Lemma \ref{Csmlemma} require $\C$ to contain all the spaces $\prod_{i=1}^r\GL_{n_i}(\sO_X^{\an}(U))$ for analytic open subsets $U$ of complex manifolds $X$. We also need some form of exponential law for objects of $\C$. 

Since we only use smoothness to generate differential forms and contracting homotopies on symmetric spaces, it is possible that real analytic morphisms will serve equally well, \emph{mutatis mutandis}. 

\end{remark}

\begin{definition}
 Let $\Ch_{\R}$ be the category of  (unbounded) chain complexes in real vector spaces. For any small category $I$, write $\Ch_{\R}(I)$ for the category of presheaves in real chain complexes on $I$ (i.e. functors $I^{\op} \to \Ch_{\R}$).
\end{definition}

\begin{proposition}\label{pshfmodelstr}
 For any small category $I$, there is a  cofibrantly generated model structure on the category  $\Ch_{\R}(I)$, with a morphism $f \co A \to B$ being a fibration (resp. weak equivalence) whenever the maps $f_i \co A(i) \to B(i)$ are surjections (resp. quasi-isomorphisms) for all $i \in I$.
\end{proposition}
\begin{proof}
This is essentially the projective model structure of Bousfield and Kan \cite{bousfieldkan}. 
Generating cofibrations are given by $\R.\Hom_I(-,i)[n] \to \cone(\R.\Hom_I(-,i))[n]$, and generating  trivial cofibrations by $ 0\to \cone(\R.\Hom_I(-,i)) [n]$, where $V.\Hom_I(-,i)$ is the presheaf $j \mapsto \bigoplus_{f\co j \to i}V$. 
\end{proof}

\begin{definition}
 Given Fr\'echet manifolds $Z,T$, write $\C^{\infty}(Z,T)$ for the set of smooth functions $Z \to T$. Given $V = \{V_i\}_i\in \pro(\Ban)$, write $\C^{\infty}(Z,V):= \Lim_i \C^{\infty}(Z, V_i)$, noting that this inherits the structure of a real vector space.
\end{definition}

\begin{proposition}\label{Smprop}
The functor $\Sm \co\Ch(\pro(\Ban)) \to \Ch_{\R}(\FrM) $ given by 
\[
 \Sm(V)(Z)_n:= \C^{\infty}(Z,V_n)
\]
is right Quillen.  
\end{proposition}
\begin{proof}
 First, we need to show that $\Sm$ has a left adjoint. Given $Z \in \FrM$, consider the vector space $\R.Z^{\delta}$ with basis given by the points of $Z$. Define $I$ to be the set of equivalence classes of seminorms $\nu$ on $\R.Z^{\delta}$ for which the induced map $Z \to  (\R.Z^{\delta})_{\nu}$ is smooth, where the subscript $\nu$ denotes completion with respect to $\nu$. 

Then $I$ is a directed set with order relation $[\nu] \le [\nu']$ whenever $\nu \le C\nu'$ for some $C>0$,  upper bounds existing because $[\nu+\nu'']\ge [\nu],[\nu'']$.  We then have a pro-Banach space $\ban(\R.Z):= \{(\R.Z^{\delta})_{\nu}\}_{\nu \in I}$. Clearly, $\Hom_{\pro(\Ban)}( \ban(\R.Z),V) \cong \C^{\infty}(Z,V)$ for all Banach spaces $V$. Extending the functor $\ban$ to the whole of $\Ch_{\R}(\FrM) $ by suspension and left Kan extension then gives a left adjoint to $\Sm$. 

We now need to check that $\Sm$ preserves (trivial) fibrations. It suffices to verify this on the cogenerators, since $\Sm$ preserves all limits. It is automatic that $\cone(\C^{\infty}(Z,\ell^{\infty}(I)))[n] \to \C^{\infty}(Z,\ell^{\infty}(I))[n-1]$ is surjective for all $Z,I,n$, and that  $\cone(\C^{\infty}(Z,\ell^{\infty}(I)))[n] \to 0$ is a surjective quasi-isomorphism.
\end{proof}

\begin{definition}\label{Hsmdef}
 Write $\oL\ban \co \Ch_{\R}(\FrM) \to \Ch(\pro(\Ban))$ for the derived pro-Banach completion functor, given by composing $\ban$ with cofibrant replacement, and write $\oR \Sm$ for the composition of $\Sm$ with fibrant replacement. Given $U \in \Ch_{\R}(\FrM), V \in\Ch(\pro(\Ban))$, define
\[
 \oR\HHom_{\dif}(U,V):= \oR\HHom_{\Ban}( \oL \ban U, V) \simeq   \oR\HHom_{\FrM}(  U, \oR \Sm V), 
\]
and write
\[
 \Ext^i_{\dif}(U,V):= \H^i\oR\HHom_{\dif}(U,V).
\]
\end{definition}

\subsection{Derived $\Hom$-spaces and homotopy ends}

\begin{definition}
 For any small category $J$, the functor $\HHom_J \co \Ch_{\R}(J)^{\op} \by \Ch_{\R}(J) \to \Ch_{\R}$ is given by 
\[
\HHom_J(U,V)_n := \prod_i \Hom_{J}(U_i, V_{i+n}),
\]
with differential $df: = d_V \circ f + (-1)^i f \circ d_U$. For the singleton category $*$, we simply write $\HHom:=\HHom_*$.
\end{definition}

Observe that $\HHom_J$ is right Quillen for the model structure of Proposition \ref{pshfmodelstr}, so has a right-derived functor $\oR\HHom_J$. 

\begin{definition}
Given a small category $I$ and a functor $F\co I\by I^{\op} \to \Ch_{\R}$, define the homotopy end of $F$ by
\[
 \int^h_{i \in I} F(i,i):= \oR\HHom_{ I\by I^{\op} }(\R.\Hom_I, F),
\]
where $\R.\Hom_I$ denotes the presheaf sending $(i,j)$ to the real vector space with basis  $\Hom_I(i,j)$.
\end{definition}

The following lemma, which is almost tautological, allows us to rewrite derived $\Hom$-spaces as homotopy ends, which will considerable simplify their manipulation.
\begin{lemma}
There is a canonical equivalence 
\[
 \oR\HHom_I(U,V) \simeq \int^h_{i \in I} \HHom( U(i),V(i)),
\]
for all $U,V \in \Ch_{\R}(I)$.
\end{lemma}
\begin{proof}
The right-derived functor $\oR F$ of a functor $F$ can be characterised as the universal functor under $F$ preserving weak equivalences.
 By definition, we have
\[
  \int^h_{i \in I} \HHom( U(i),V(i))= \oR\HHom_{ I\by I^{\op} }(\R.\Hom_I, \HHom(U,V)),
 \]
which is the evaluation at $(U,V)$ of the right-derived functor of the functor 
\[
 \HHom_{ I\by I^{\op} }(\R.\Hom_I, \HHom(-,-))
\]
on $\Ch_{\R}(I)^{\op} \by \Ch_{\R}(I)$.

On the other hand,
\[
 \HHom_{ I\by I^{\op} }(\R.\Hom_I, \HHom(U,V))= \HHom_I(U,V),
\]
with corresponding right-derived functor $(U,V) \mapsto \oR\HHom_I(U,V) $. 
\end{proof}

\begin{definition}
 Given a small category $I$, define $\Ch(\pro(\Ban),I)$ to be the category of functors $I^{\op} \to \Ch(\pro(\Ban))$.
\end{definition}

We will not attempt to put  a model structure on this category; since the model structure of Proposition \ref{banmodelprop} is only fibrantly cogenerated, we cannot take a suitable model structure off the shelf. 
The category $\Ch(\pro(\Ban),I)$ has an obvious notion of levelwise strict quasi-isomorphism, so we have a relative category, and hence notions of right-derived functors and an $\infty$-category given by localisation, but we will save space by making the following a definition.

\begin{definition}\label{RHomBan}
Define  $\oR\HHom_{I,\Ban}\co \Ch(\pro(\Ban),I)^{\op} \by \Ch(\pro(\Ban),I) \to \Ch_{\R}$ by
\[
 \oR\HHom_{I,\Ban}(U,V):= \int^h_{i \in I} \oR\HHom_{\Ban}(U(i),V(i)).
\]
Here, $\oR\HHom_{\Ban} $ is the right derived functor of the right Quillen functor $\HHom_{\Ban}$ from Definition \ref{HHomBandef} with respect to the model structure of Proposition \ref{banmodelprop}. Explicitly, $ \oR\HHom_{\Ban}(U(i),V(i))\simeq \HHom_{\Ban}(U(i),V(i)')$ for a fibrant replacement $V(i) \to V(i)'$.

For the category $\Op(X_{\an})$ of open subsets of a complex manifold $X_{\an}$, we just denote $\oR\HHom_{\Op(X_{\an}),\Ban}$ by $\oR\HHom_{X_{\an},\pro(\Ban) }$.
\end{definition}

\subsection{Pro-Banach hypersheaves and $K_{\Ban,>0}$}\label{Kbansn}

\begin{definition}
Given a  spectrum $Y=\{Y^n\}$ in simplicial sets in the sense of \cite[Definition 2.1]{BousfieldFriedlander}  and a $\Q$-algebra $A$, write $Y_A$ for the chain complex
\[
 Y_A:=  \LLim_n N \bar{\CC}_{\bt}(Y^n,A)[n]
\]
of $A$-modules, where $\bar{\CC}$ denotes reduced 
chains on a pointed simplicial set, and $N$ the Dold--Kan normalisation;  the maps
\[
 N\bar{\CC}_{\bt}(Y^n,A)\to N\bar{\CC}_{\bt}(Y^{n+1},A)[1]
\]
combine the structure map 
\[
 (N^{-1}A[-1])\ten_A \bar{\CC}_{\bt}(Y^n,A)\cong \bar{\CC}_{\bt}(S^1 \wedge Y^n,A)  \to\bar{\CC}_{\bt}(Y^{n+1},A)
\]
with the Eilenberg--Zilber shuffle map
\[
 A[-1]\ten_A N\bar{\CC}_{\bt}(Y^n,A) \to N( (N^{-1}A[-1])\ten_A \bar{\CC}_{\bt}(Y^n,A)). 
\]
\end{definition}

Since the rational stable Hurewicz map $\pi_*(Y)\ten \Q \to \H_*(Y,\Q)$ is an isomorphism, 
this construction satisfies $\H_i(Y_A) =\pi_i(Y)\ten_{\Z}A$, in particular sending stable equivalences to quasi-isomorphisms. It is equivalent to smashing with the Eilenberg--MacLane spectrum $HA$. The  right adjoint to the functor $Y \mapsto Y_A$ sends a chain complex $V$ to the spectrum $\{N^{-1}\tau_{\ge 0}(V[-n])\}$, where $\tau$ denotes good truncation.

\begin{definition}
Given a Fr\'echet algebra $A$,  define    $\fK_{>0}(A)\in \Ch_{\R}(\FrM)$  by $Z \mapsto K_{>0}(\C^{\infty}(Z, A))_{\R}$. Varying $U$ then yields a   $\Ch_{\R}(\FrM)$-valued presheaf $\fK_{>0}(\sO_X^{\an})$ on any complex manifold $X_{\an}$, given by
  \[
   U \mapsto \fK_{>0}(\sO_X^{\an}(U)).
  \]
 \end{definition}

\begin{definition}\label{ChBanXdef}
 Define $\Ch(\pro(\Ban),X_{\an})$ to be the category of presheaves  of  complexes of pro-Banach spaces on $X_{\an}$.
\end{definition}

\begin{definition}\label{KBandef}
Given a  Fr\'echet algebra $A$, we define $ K_{\Ban,>0}(A)\in \Ch(\pro(\Ban))$ by
\[
 K_{\Ban,>0}(A):=\oL\ban\fK_{>0}(A).
\]
This yields a  pro-Banach presheaf $K_{\Ban,>0}(\sO_X^{\an})\in \Ch(\pro(\Ban),X_{\an})$ given  by
\[
 K_{\Ban,>0}(\sO_X^{\an})(U)=K_{\Ban,>0}(\sO_X^{\an}(U)).
\]
 \end{definition}

The remainder of the paper will be dedicated to showing that  $K_{\Ban,>0}(\sO_X^{\an})$ is locally strictly quasi-isomorphic to the real Deligne complex.

\begin{remark}
 The description of smooth morphisms between Fr\'echet manifolds in \cite[Lemma to Theorem 3.12]{KrieglMichor} gives an embedding in the category of acts for the monoid $\C^{\infty}(\R,\R)$, by sending a manifold $M$ to the $\C^{\infty}(\R,\R)$-act $\C^{\infty}(\R,M) $. Thus  $ K_{\Ban,>0}(A)$ is an invariant of the abstract ring $\C^{\infty}(\R,A)$ equipped with its $\C^{\infty}(\R,\R)$-action.
\end{remark}

\begin{definition}\label{Banhypshf}
A presheaf $\sV\in \Ch(\pro(\Ban),X_{\an})$ is said to be a hypersheaf if for every open hypercover $\tilde{U}_{\bt} \to U$, the map
\[
 \sV(U) \to \ho\Lim_{n \in \Delta} \sV(\tilde{U}_n)
\]
to the homotopy limit  
is a strict quasi-isomorphism.
\end{definition}

\begin{definition}\label{localstrictdef}
 A morphism $f\co \sE \to \sF$ in $\Ch(\pro(\Ban),X_{\an}) $ is said to be a local strict quasi-isomorphism if for all hypersheaves $\sV$ in $\Ch(\pro(\Ban),X_{\an}) $, the map
\[
 \oR\HHom_{X_{\an}, \Ban}(\sF, \sV) \to  \oR\HHom_{X_{\an}, \Ban}(\sE, \sV)
\]
is a quasi-isomorphism, for $\oR\HHom_{X_{\an}}$ as in Definition \ref{RHomBan}.
\end{definition}

\subsection{Descent and polydiscs}\label{polydiscssn}

\begin{definition}\label{BXandef}
 Define $\cB(X_{\an})$ to be the category whose objects  are those open subsets of $X_{\an}$ which are isomorphic to open polydiscs, i.e.
\[
 \{ z \in \Cx^d\,:\, |z_i|<1 \quad \forall i\},
\]
and whose morphisms are inclusions.
\end{definition}

We now check that polydiscs determine everything we will need to consider.
\begin{lemma}\label{bifunctordescent}
Given a bifunctor $\sF\co \Op(X_{\an})^{\op}\by \Op(X_{\an})\to \Ch_{\R}$ such that $\sF(-,U) \in \Ch_{\R}(X_{\an})$ is a hypersheaf for all $U$, the map  
\[
 \int^h_{U \in \Op(X_{\an})} \sF(U,U) \to \int^h_{U \in \cB(X_{\an})}\sF(U,U)
\]
of homotopy ends is a quasi-isomorphism.
\end{lemma}
\begin{proof}
Writing $h_{\cB}$ for the set-valued presheaf $\Hom_{\cB(X_{\an})}(-,-)$ on $\cB(X_{\an}) \by \cB(X_{\an})^{\op}$ and similarly for $h_{\Op}$, we begin by observing that a cofibrant replacement for $h_{\cB}$ as a simplicial set-valued presheaf is given by 
\[
 \tilde{h}(U,V):=  B(U \da \cB(X_{\an})\da V),
\]
the nerve of  the  category  $(U \da \cB(X_{\an}) \da V)$  of objects under $U$ and over $V$,  consisting of diagrams $U \to W \to V$ in $\cB(X_{\an})$. This follows because the category $(U \da \cB(X_{\an}) \da V)$ is either empty (when $U \nsubseteq V$) or has final object $U \to V \to V$, so $ \tilde{h}(U,V)$ is either empty or contractible and $\tilde{h}(U,V) \to h_{\cB}(U,V)$ is a weak equivalence. That the presheaf $\tilde{h}$ is cofibrant follows because
\[
 \Hom_{(\cB(X_{\an})^{\op} \by \cB(X_{\an})) }(\tilde{h}_n, \sF) \cong \prod_{\substack{W_0 \to \ldots \to W_n\\ \in B_n\cB(X_{\an}) }} \cF(W_0,W_n).
\]
Thus the normalised $\R$-linearisation $N\R.\tilde{h}(-,-)= N\CC_{\bt}( \tilde{h}(-,-), \R) $ is  a cofibrant replacement for $\R.h_{\cB}(-,-)$ in $\Ch_{\R}(\cB(X_{\an})^{\op} \by \cB(X_{\an}))$.

Writing $i \co \cB(X_{\an})\to \Op(X_{\an})$ for the inclusion functor, 
we then have
\begin{eqnarray*}
 \int^h_{U \in \cB(X_{\an})}\sF(U,U)&=& \oR\HHom_{\cB(X_{\an})^{\op}\by \cB(X_{\an}) }(\R.h_{\cB}, (i,i)^*\sF)\\
&\simeq& \HHom_{\cB(X_{\an})^{\op}\by \cB(X_{\an}) }(N\R.\tilde{h}, (i,i)^*\sF)
\end{eqnarray*}

Now, the functor $(i,i)^*$ is right Quillen, with a left adjoint $(i,i)_!$ determined by
\begin{align*}
 (i,i)_!&\R(\Hom_{\cB(X_{\an})}(-,A)\by \Hom_{\cB(X_{\an})}(B,-)) \\
&= \R(\Hom_{\Op(X_{\an})}(-,A)\by \Hom_{\Op(X_{\an})}(B,-))
\end{align*}
and left Kan extension. By adjunction, we then have
\[
 \int^h_{U \in \cB(X_{\an})}\sF(U,U) \simeq \HHom_{\cB(X_{\an})^{\op}\by \cB(X_{\an}) }( (i,i)_!N\R.\tilde{h}, \sF).
\]

Next, observe that $((i,i)_!N\R.\tilde{h})(U,V) =  N\R.B(U \da \cB(X_{\an})\da V)$, for $U,V$ now in $\Op(X_{\an})$, rather than just in $\cB(X_{\an})$. There is a canonical map 
\[
 (i,i)_!N\R.\tilde{h} \to \R.h_{\Op},
\]
and we now wish to say that for fixed $V \in \Op(X_{\an})$, the map
\[
\phi_V\co  (i,i)_!N\R.\tilde{h}(-,V) \to\R.h_{\Op}(-,V)
\]
is a local quasi-isomorphism.
This follows from the observation above that for all $U \in \cB(X_{\an})$, the map $\tilde{h}(U,iV)\to \Hom_{\cB(X_{\an})}(U,V)$ is a weak equivalence, $i$ being full and faithful.


Writing
\[
 \oR\HHom_{\Op(X_{\an})^{\op}\by \Op(X_{\an}) }( h_{\Op}, \sF)\simeq \int^h_{V \in \Op(X_{\an})} \oR\HHom_{\Op(X_{\an})}(h(-,V), \sF(-,V)),
\]
and using that $\sF(-,V)$ is a hypersheaf, we then get
\begin{eqnarray*}
 \int^h_{U \in \Op(X_{\an})} \sF(U,U) &\simeq& \int^h_{V \in \Op(X_{\an})} \oR\HHom_{\Op(X_{\an})}(\R.h(-,V), \sF(-,V))\\
&\simeq& \int^h_{V \in \Op(X_{\an})} \oR\HHom_{\Op(X_{\an})}( (i,i)_!N\R.\tilde{h}(-,V), \sF(-,V))\\
&\simeq& \oR\HHom_{\Op(X_{\an})^{\op}\by \Op(X_{\an}) }( \oL(i,i)_!\R.h_{\cB}, \sF)\\
&\simeq& \oR\HHom_{\cB(X_{\an})^{\op}\by \cB(X_{\an}) }(\R.h_{\cB}, (i,i)^*\sF)\\
&\simeq& \int^h_{U \in \cB(X_{\an})} \sF(U,U).
\end{eqnarray*}

\end{proof}


The following lemma gives a statement for strict quasi-isomorphisms whose analogue for quasi-isomorphisms is well-known.

\begin{lemma}\label{localstrictlemma}
If $f \co\sE \to \sF$ is a  morphism in $\Ch(\pro(\Ban),X_{\an})$ for which the maps $\sE(U) \to \sF(U)$ are strict quasi-isomorphisms for all open polydiscs $U \subset X_{\an}$, then $f$ is a local strict quasi-isomorphism in the sense of Definition \ref{localstrictdef}.
\end{lemma}
\begin{proof}
 For every pro-Banach hypersheaf $\sV$, we need to show that
\[
 \oR\HHom_{X_{\an}, \Ban}(\sF, \sV) \to  \oR\HHom_{X_{\an}, \Ban}(\sE, \sV)
\]
is a quasi-isomorphism. We may write 
\[
  \oR\HHom_{X_{\an}, \Ban}(\sF, \sV) \simeq \int^h_{U \in \Op(X_{\an})}\oR\HHom_{\Ban}(\sF(U), \sV(U)).
\]
Since $\sV$ is a hypersheaf, the functor  $U \mapsto  \oR\HHom_{\Ban}(W, \sV(U))$ is a hypersheaf for all pro-Banach complexes $W$, and thus the bifunctor
\[
 \oR\HHom_{\Ban}(\sF(-), \sV(-))
\]
satisfies the conditions of Lemma \ref{bifunctordescent}, so
\[
 \oR\HHom_{X_{\an}, \Ban}(\sF, \sV) \simeq \int^h_{U \in \cB(X_{\an})}\oR\HHom_{\Ban}(\sF(U), \sV(U)).
\]

The required quasi-isomorphism now follows because the maps
\[
 \oR\HHom_{\Ban}(\sF(U), W) \to \oR\HHom_{\Ban}(\sE(U), W)
\]
are quasi-isomorphisms for all $U \in \cB(X_{\an})$ and all pro-Banach complexes $W$.
\end{proof}

\section{Equivalent descriptions of $K_{\Ban,>0}$}\label{Kequivsn}

\subsection{Differentiable cohomology}
We now develop some  analogues for infinite-dimensional Lie groups  of the results of  \cite[\S\S 4--5]{HochschildMostow}, but in a more \emph{ad hoc} fashion, with less general results.

\begin{definition}
 Given a Fr\'echet Lie group $G$ whose underlying manifold lies in $\FrM$, and   $V \in \Ch_{\R}(\FrM)$, define
\[
 \CC^{\bt}_{\dif}(G, V)
\]
to be the product total complex of the Dold--Kan conormalisation of the cosimplicial diagram
\[
\xymatrix@1{ V(1) \ar@<.5ex>[r] \ar@<-.5ex>[r] & V(G)  \ar@<.5ex>[r] \ar[r] \ar@<-.5ex>[r] & V(G \by G) \ldots },
\]
with maps coming from the maps in the nerve $BG$.

For $V \in \Ch(\pro(\Ban))$, we simply write $\CC^{\bt}_{\dif}(G, V):=\CC^{\bt}_{\dif}(G, \Sm(V)) $, for the functor $\Sm$ of Proposition \ref{Smprop}.
\end{definition}

\begin{definition}
Given a simplicial Fr\'echet manifold $Y_{\bt}$, write $N\R.Y_{\bt} \in \Ch_{\R}(\FrM)$ for the presheaf
\[
 Z \mapsto N \R.\C^{\infty}(Z,Y_{\bt}),
\]
where $\R.T$ denotes the real vector space with basis $T$, and $N$ is Dold--Kan normalisation.
\end{definition}

\begin{lemma}\label{Csmlemma}
If $G$ is a Fr\'echet Lie group whose underlying space lies in $\FrM$, 
then for all $V \in \Ch_{\R}(\FrM)$ there is a  canonical quasi-isomorphism
\[
 \CC^{\bt}_{\dif}(G, V) \to \oR\HHom_{\FrM}(N\R.BG, \Sm(V)).
\]
\end{lemma}
\begin{proof}
The complex $ \CC^{\bt}_{\dif}(G, V)$ is just 
\[
 \HHom_{\FrM}(N\R.BG, \Sm(V))= \HHom_{\FrM}(N\CC_{\bt}(BG(-),\R), \Sm(V)),
\]
 and every object of $ \Ch_{\R}(\FrM)$ is fibrant, so it suffices to show that $N\R.BG $ is cofibrant. The Dold--Kan normalisation $N\R.BG $ is a retract of the total chain complex associated to the simplicial presheaf $\R.BG $. Brutal truncation of the total chain complex gives a filtration $W_0, W_1,\ldots$ whose quotients are the presheaves $ \R.G^n[-n]$, so $W_{n-1} \to W_n$ is a pushout of a generating cofibration, and is therefore a cofibration. Then $ N\R.BG$  is a retract of $\LLim_n W_n$, so  is also cofibrant.
\end{proof}

\subsection{Symmetric spaces and Relative Lie algebra cohomology}

For $U \subset X_{\an}$ an open subset, observe that $\GL_n(\sO_X(U))$ is a closed subset of a Fr\'echet space: this follows because $ \sO_X(U)$ is a Fr\'echet algebra, with the system of seminorms given by sup norms on compact subsets of $U$. Explicitly, when $U\cong D^m$ is a unit polydisc, we have
\[
\sO(D^m)= \Lim_{r<1} \ell^{1}(\{(z_1/r)^{i_1}(z_2/r)^{i_2}\ldots (z_m/r)^{i_m}\}_{i \in \N_0^m}),
\]
for the co-ordinate functions $z_1, \ldots, z_m$ on $D^m$. 
Therefore  $\GL_n(\sO_X(U)) \subset M_n(\sO_X(U))$ can be characterised as 
  the closed subspace 
 \[
  \{(g, \lambda) \in \Mat_n(\sO(D^m) )\by \sO(D^m) ~:~ \lambda\det(g)=1\}.
 \]

Evaluation at a basepoint $x \in U$ gives a surjection $\GL_n(\sO_X(U))\to \GL_n(\Cx)$, and hence a diffeomorphism $\GL_n( \sO_X(U))\cong \GL_n(\Cx) \ltimes
\ker(\GL_n\sO_X(U)\xra{x^*}\GL_n(\Cx))$ given by multiplication. Proposition \ref{contractibleprop} below shows that $ \GL_n(\sO_X(U))$ is in fact a Fr\'echet manifold, so in particular  
\[
 U_n\backslash \GL_n(\sO_X(U))\cong (U_n \backslash  \GL_n(\Cx)) \by\ker(\GL_n\sO_X(U)\xra{x^*}\GL_n(\Cx))
\]
is  also a Fr\'echet manifold. 


Similarly, for any  Fr\'echet algebra $A$, we have that $\GL_n(A)$ is a closed subset of a Fr\'echet space. If $A$ is a Banach algebra, then $\GL_n(A)\subset M_n(A)$ is an open subset of a Banach space, hence a Banach manifold. It follows that for any multiplicatively convex Fr\'echet algebra $A$, the space $\GL_n(A)$ is an inverse limit of Banach manifolds, but in general it will not be open in $M_n(A)$. 

 If $A$ has the structure of an augmented Fr\'echet $R$-algebra for  some finite-dimensional real algebra  $R$, and $K$ is a Lie subgroup of $\GL_n(R)$, then the reasoning above shows that $K\backslash \GL_n(A)$ is  a Fr\'echet manifold whenever $\GL_n(A)$ is so.

\begin{proposition}\label{contractibleprop}
 For $B \subset X_{\an}$ an open polydisc, the space $ \GL_n(\sO_X^{\an}(B))$ is a Fr\'echet manifold, and the morphism $U_n \to G:=\GL_n(\sO_X^{\an}(B))$
is a smooth $U_n$-equivariant deformation retract, in the sense that there exists a $\C^{\infty}$ morphism, equivariant with respect to the  left $U_n$-action,
\[
 h \co G \by \R \to G
\]
with $h(-,1)$ the identity map and $h(-,0)$  a retraction onto $U_n$.

In particular, the Fr\'echet manifold
$ M:= U_n \backslash \GL_n(\sO_X^{\an}(B))$ is smoothly contractible.
\end{proposition}
\begin{proof}
If we write $D$ for the open unit disc in $\Cx$, then $B$ is isomorphic to the $d$-fold product  $D^d$. The projections $D^d \to D^{d-1} \to \ldots \to D \to \bt$ then induce morphisms
\[
 \GL_n(\Cx) \to \GL_n(\sO(D)) \to  \GL_n(\sO(D^2))\to \ldots \to \GL_n(\sO(D^d)),
\]
 where $\sO(U)$ denotes the ring of complex-analytic functions on $U$.

Moreover, we have a diffeomorphism 
\[
 \GL_n(\sO(D^m)) \cong \GL_n(\sO(D^{m-1}))\by \{g \in\GL_n(\sO(D^m))\,:\, g(z_1, \ldots, z_{m-1},0)=I\},
\]
with $g$ mapping to $(g(z_1, \ldots, z_{m-1},0), g(z_1, \ldots, z_{m-1},0)^{-1}g)$.

If we write $d_mg:= \frac{\pd g}{\pd z_m} dz_m$ and $\delta_mg:= g^{-1}d_mg $, we get a smooth map
\[
\delta_m\co \GL_n(\sO(D^m)) \to \gl_n \ten \sO(D^m)dz_m ,
\]
whose fibres are the orbits of the left multiplication by $\GL_n(\sO(D^{m-1})) $. 

A smooth inverse to 
\[
 \delta_m \co \{g \in\GL_n(\sO(D^m))\,:\, g(z_1, \ldots, z_{m-1},0)=I\} \to  \gl_n \ten \sO(D^m)dz_m
\]
 is given by iterated integrals. Explicitly, we   send $\omega$ to the function $\iota(\omega)$ given at a point $(z_1, \ldots, z_m)$ by
\[
 I+ \sum_{k >0} \int_{1 \ge t_1 \ge \ldots \ge t_k \ge 0}  (\gamma^*\omega)(t_k) \cdots
(\gamma^*\omega)(t_1), 
\]
for the path $\gamma(t)= ( z_1, \ldots, z_{m-1},t z_m)$.
To see that this converges, consider a closed polydisc $\bar{D}^m(r)$ of radius $r<1$. Since $\bar{D}^m(r)$ is compact, $\omega$ is bounded here, by $C_r$ say, so $\gamma^*\omega$ is bounded by $rC_r $, giving
\[
 \|\iota(\omega)\|_{\bar{D}^m(r)} \le \sum_{k\ge 0} (rC_r)^k/k! = \exp(rC_r).
\]

We therefore have a diffeomorphism
\[
 \GL_n(\sO(D^m)) \cong \GL_n(\sO(D^{m-1}))\by (\gl_n \ten \sO(D^m)dz_m).
\]
Proceeding inductively, we obtain a $\GL_n(\Cx)$-equivariant diffeomorphism
\[
 \GL_n(\sO(D^d)) \cong \GL_n(\Cx) \by (\gl_n \ten\prod_{m=1}^d \sO(D^m)dz_m),
\]
which in particular ensures that $\GL_n(\sO(D^d))$ is a Fr\'echet manifold.

Since $\sO(D^m)$ is smoothly contractible, it follows that 
\[
 \GL_n(\Cx) \into \GL_n(\sO(D^d))
\]
is a smooth $\GL_n(\Cx)$-equivariant deformation retract. Because $U_n \to  \GL_n(\Cx)$  is a smooth $U_n$-equivariant deformation retract (taking $h(g,t)= g (g^{\dagger}g)^{(t-1)/2}$), the result follows.
\end{proof}

\begin{remarks}\label{smoothDeligneRmk}
 In fact, the formulae in the proof of Proposition \ref{contractibleprop} ensure that the homotopy  is real analytic, once we note that  the isomorphism $g \mapsto \log(g^{\dagger}g)$  between $U_n\backslash \GL_n(\Cx)$ and positive semi-definite Hermitian matrices is so.

The proof of Proposition \ref{contractibleprop} also adapts to show that $U_n \to \GL_n(\C^{\infty}(\R^d,\Cx))$ (resp.  $O_n \to \GL_n(\C^{\infty}(\R^d,\R))$)  is a   smooth $U_n$-equivariant (resp. $O_n$-equivariant) deformation retract of Fr\'echet manifolds, by replacing holomorphic differentials with smooth differentials. 
\end{remarks}

\begin{definition}
Given a Fr\'echet manifold $Y$ and a $c^{\infty}$-complete real vector space $V$, define
\[
 A^{\bt}(Y,V)
\]
 to be the de Rham complex of $Y$ with coefficients in $V$ (denoted $\Omega^{\bt}(Y,V)$ in \cite[Remark 33.22]{KrieglMichor}), equipped with its de Rham differential. Explicitly, $A^m(Y,V) $ consists of smooth sections of the bundle $L^m_{\mathrm{alt}}(TY, V)$ of bounded alternating $m$-linear functions from the kinematic tangent space of $Y$ to $V$, with the topology of \cite[Definition 3.11]{KrieglMichor}. 
Note that $A^0(Y,V)=\C^{\infty}(Y,V)$.
\end{definition}

\begin{proposition}\label{symmspaceprop1}
For $M$  a smoothly contractible  Fr\'echet manifold,   any real  Banach space $V$, and any  Fr\'echet manifold $Y$, the inclusion morphism 
\[
 A^0(Y,V) \to  A^0(Y, A^{\bt}(M,V))
\]
is a quasi-isomorphism of real cochain complexes.
\end{proposition}
\begin{proof}
We now follow the proof of the Poincar\'e lemma in \cite[33.20, 34.2]{KrieglMichor}.

The contracting  homotopy $h \co M \by \R \to M$ 
gives a  linear map 
\[
 h^* \co A^{\bt}(M,V) \to A^{\bt}(M\by \R,V),
\]
which is smooth by the exponential law. We then define a smooth map $I$ on 
\[
 A^k(M\by \R,V) \cong A^0(\R, A^k(M,V)) \oplus A^1(\R, A^{k-1}(M,V))
\]
to be $0$ on the first factor and 
\[
 \int_{[0,1]}\co A^1(\R, A^{k-1}(M,V)) \to A^{k-1}(M,V)
\]
on the second factor; this exists by  \cite[Proposition 2.7]{KrieglMichor} because $A^{k-1}(M,V)$ is a $c^{\infty}$-complete vector space, as observed in  \cite[Remark 33.22]{KrieglMichor}.

Setting $k := I \circ h^*$ gives a smooth linear map $A^k(M,V) \to A^{k-1}(M,V)$ which acts as a contracting homotopy for the complex $ V \to A^{\bt}(M,V)$. We therefore have contracting homotopies on applying $A^0(Y,-)$, giving the required quasi-isomorphisms.
\end{proof}

\begin{definition}
 Given a Fr\'echet manifold $Y$, write $EY$ for the  simplicial manifold given by 
$
 (EY)_m := Y^{m+1}
$,
with operations 
\begin{align*}
\pd_i(y_0, \ldots, y_m) &= (y_0,  \ldots,y_{i-1}, y_{i+1}, \ldots,  y_m),\\
\sigma_i(y_0, \ldots, y_m) &= (y_0,  \ldots,y_i, y_i, \ldots,  y_m).
\end{align*}
\end{definition}

Note that for a Fr\'echet Lie group $G$, the diagram $EG$ admits a right action componentwise by $G$, with quotient isomorphic to the nerve $BG$.

\begin{definition}
 Given a simplicial  Fr\'echet manifold $Y_{\bt}$  and a $c^{\infty}$-complete vector space $E$, define 
\[
 A^0(Y_{\bt},E)
\]
to be the cochain complex  given by  Dold--Kan conormalisation of 
\[
\xymatrix@1{ A^0(Y_0,E)\ar@<.5ex>[r] \ar@<-.5ex>[r] & A^0(Y_1,E)   \ar@<.5ex>[r] \ar[r] \ar@<-.5ex>[r] & A^0(Y_2,E ) \ldots }.
\]
\end{definition}

\begin{corollary}\label{symmspacecor1}
Take a smoothly  contractible Fr\'echet manifold $M$ equipped with smooth action of a Fr\'echet Lie group $G$.
Then there is a canonical quasi-isomorphism of real cochain complexes
\[
 \CC^{\bt}_{\dif}(G, V) \to \Tot A^0(EG, A^{\bt}(M, V))^G
\]
to $G$-invariants in the total complex of the double complex,
where the  $G$-action is given by combining the right actions of $G$ on $EG$ and on $M$.
\end{corollary}
\begin{proof}
 
By Proposition \ref{symmspaceprop1}, the maps 
\[
A^0(G^m,V) \to A^0(G^m,A^{\bt}(M,V) )
\]
are all quasi-isomorphisms. Since $ A^0((EG)_m,  V)^G \cong A^0(G^m,V)$, these combine to give a quasi-isomorphism
\[
 \CC^{\bt}_{\dif}(G, V)= A^0(EG,  V)^G  \to \Tot A^0(EG, A^{\bt}(M, V))^G.
\]
\end{proof}

\begin{definition}
Given Fr\'echet spaces $U,V$, define $U \hten V$ to be the projective tensor product of $U$ and $V$. This is a Fr\'echet space with the property that maps from $U\hten V$ to pro-Banach spaces $W$ correspond to continuous bilinear maps $U\by V \to W$.

Write  $V^{\hten n}:= \overbrace{V\hten V \hten \ldots \hten V}^n$, $ \hat{\Symm}^nV:= V^{\hten n}/ \Sigma_n$ and  $\hat{\Lambda}^n V:= V^{\hten n}\ten_{\R[\Sigma_n]}\sgn$, where  $\sgn$ is the one-dimensional real vector space on which $\Sigma_n$ acts by the signature.
\end{definition}

By \cite[Proposition 5.8]{KrieglMichor}, these operations coincide with the corresponding bornological tensor operations. 
Similarly, we can define tensor operations on bundles over a manifold.  

\begin{lemma}\label{integratelemma}
If $E$ is  a $c^{\infty}$-complete (cf. \cite[Theorem 2.14]{KrieglMichor})  real  vector space, $K$ a  compact finite-dimensional manifold with smooth measure $\mu$, and $f \co K \to E$ is smooth, then the integral 
\[
 \int_K f d\mu
\]
exists. Moreover, for all continuous linear functionals $\ell$ on $E$, we have $\ell (\int_K f d\mu) = \int_K (\ell \circ f) d\mu$.   
\end{lemma}
\begin{proof}
By  \cite[Lemma 2.4]{KrieglMichor}, there is a natural topological embedding of $E$ into the space of 
all linear functionals on $E^*$ which are
bounded on equicontinuous sets. It therefore suffices to show that the functional
\[
 \ell \mapsto \int_K (\ell \circ f) d\mu
\]
on $E^*$ is given by  a (necessarily unique) element of $E$.

The question is local on $K$, so we need only show that for any smooth $g\co \R^n \to E$, the integral
\[
 \int_{[0,1]^n} g dx_1\wedge \ldots \wedge d x_n
\]
lies in $E$. This follows by repeated application of \cite[Proposition 2.7]{KrieglMichor}, the spaces $\C^{\infty}(\R^m,E)$ being $c^{\infty}$-complete by \cite[Lemma 27.17]{KrieglMichor}.
\end{proof}

\begin{proposition}\label{cohoderhamprop0} 
If $A$ is a  Fr\'echet algebra $A$ for which  $G:=\GL_n(A)$
is a Fr\'echet manifold    admitting a smooth $K$-equivariant deformation retract to a   compact finite-dimensional Lie group $K \subset G$, then for any real Banach space $V$,  there is a canonical zigzag of quasi-isomorphisms 
\[
 \CC^{\bt}_{\dif}(G,V) \simeq A^{\bt}(K\backslash G, V)^{G},
\]
noting that $K\backslash G$ has a natural Fr\'echet manifold structure because $K \to G$ is a retract.
\end{proposition}
\begin{proof}
Write  $M:= K\backslash G$,  $\g:= \gl_n(A)$ (a Fr\'echet Lie algebra), and write $\fk$ for the Lie algebra of $K$.
By Corollary \ref{symmspacecor1}, we need to show that the map 
\[
g \co A^{\bt}(M,V)^G \to   \Tot A^0(EG, A^{\bt}(M, V))^G
\]
is a  quasi-isomorphism.
%

The kinematic tangent bundle $TM$ of $M$ is given by the quotient $TM=(\g/\fk) \by_K G  $, where $K$ acts on $(\g/\fk) $ via the adjoint action. Let $K$ act with the left action on $G$  and with the right action on $EG$, and let $K$ act on  the first factor of $G \by EG$. Then there is a $K$-equivariant isomorphism $(G \by EG)/G \cong EG$ given by $(g, h_0, \ldots, h_m)  \mapsto (h_0g^{-1}, \ldots, h_mg^{-1})$.
 We may therefore write the map $g_i$ (the degree $i$ component of $g$) as
\begin{align*}
 A^0(G, L^i_{\mathrm{alt}}((\g/\fk),V))^{K \by G} &\to A^0(G \by EG ,L^i_{\mathrm{alt}}((\g/\fk),V))^{K \by G}\\
L^i_{\mathrm{alt}}((\g/\fk),V)^{K} &\to A^0(EG ,L^i_{\mathrm{alt}}((\g/\fk),V))^{K}.
\end{align*}

Now, the augmented simplicial diagram $EG \to *$  admits an extra degeneracy $(EG)_m \to (EG)_{m+1}$, given by 
$
 (h_0, \ldots , h_m) \mapsto (1, h_0, \ldots , h_m)  
$.
 This gives a contracting homotopy $c_i$ for 
\[
 f_i \co L^i_{\mathrm{alt}}((\g/\fk),V) \to A^0(EG ,L^i_{\mathrm{alt}}((\g/\fk),V)).
\]
It is not $K$-invariant, but since $K$ is compact and the spaces involved are all $c^{\infty}$-complete by \cite[Lemma 27.17]{KrieglMichor}, we may integrate as in Lemma \ref{integratelemma}, setting
\[
 h_i(v):= \int_{k \in K} k^{-1}c_i(kv) d \mu(k),
\]
where $\mu$ is Haar measure, normalised so that $\mu(K)=1$. Evaluation on linear functionals shows that $h$ is a $K$-equivariant contracting homotopy, so induces a quasi-isomorphism on $K$-invariants.
\end{proof}

\begin{corollary}\label{cohoderhamprop}
For $B  \in  \cB(X_{\an})$ an open polydisc, and any real Banach space $V$, there is a canonical zigzag of quasi-isomorphisms 
\[
 \CC^{\bt}_{\dif}(\GL_n(\sO(B)),V) \simeq A^{\bt}(U_n\backslash \GL_n(\sO(B)), V)^{\GL_n(\sO(B))}.
\]
\end{corollary}
\begin{proof}
 This just combines Propositions \ref{contractibleprop} and \ref{cohoderhamprop0}.
\end{proof}

\begin{definition}
Given a Fr\'echet Lie algebra $\g$, 
a closed Lie subalgebra $\fk \subset \g$ and a  $\fk$-module $V$ in pro-Banach spaces, define the continuous real Lie algebra cohomology complex $E^{\bt}_{\cts,\R}(\g, \fk;V)$ by setting
\[
 E^p_{\cts,\R}(\g, \fk;V):= \Hom_{ \pro(\Ban),\fk}(\hat{\Lambda}^p(\g/\fk), V),
\]
the space of $\fk$-linear morphisms of real pro-Banach spaces, with differential determined by
\[
d\omega(a_0, \ldots, a_p) =\sum_{i<j}(-1)^{i+j}\omega([a_i, a_j], a_0, \ldots,\hat{a}_i, \ldots ,\hat{a}_j , \ldots , a_p).
\]

Define the continuous real Lie algebra homology complex $\hat{E}_{\bt}^{\R}(\g, \fk)$ to be the chain complex of Fr\'echet spaces given by
\[
 \hat{E}_p^{\R}(\g, \fk):= \hat{\Lambda}^p(\g/\fk),
\]
with differential  determined by
\[
 d (a_0 \wedge \ldots \wedge a_p)= \sum_{i<j}(-1)^{i+j} [a_i, a_j]\wedge a_0 \wedge \ldots \wedge \hat{a}_i, \wedge \ldots \wedge \hat{a}_j \wedge \ldots \wedge a_p. 
\]
This has an $\fk$-action induced by those on $\g$ and $\fk$.

Thus $E^{\bt}_{\cts,\R}(\g, \fk;V)= \HHom_{\pro(\Ban), \fk}( \hat{E}_{\bt}^{\R}(\g, \fk),V)$, the subspace of $\fk$-invariants in $\HHom_{\pro(\Ban)}( \hat{E}_{\bt}^{\R}(\g, \fk),V)$.
\end{definition}

\begin{proposition}\label{banLiecor}
 There is a canonical zigzag of  strict quasi-isomorphisms 
\[
\oL\ban (N\R.B\GL_n(\sO_X^{\an}))\simeq \hat{E}_{\bt}^{\R}(\gl_n(\sO_X^{\an}), \fu_n)_{\fu_n}
\]
of presheaves of pro-Banach complexes on $\cB(X_{\an})$,
where $\hat{E}_{\bt}^{\R}(\gl_n(\sO_X^{\an}), \fu_n)_{\fu_n}$ is the quotient of $\hat{E}_{\bt}^{\R}(\gl_n(\sO_X^{\an}), \fu_n)$ given by coinvariants for the  action of $\fu_n$.
\end{proposition}
\begin{proof}
For all open polydiscs $U \subset X$, Proposition \ref{contractibleprop} and its preamble give a  Fr\'echet manifold structure on $U_n\backslash \GL_n(\sO_X^{\an}(U))$. 
The description of the kinematic
 tangent space in the proof of Proposition \ref{cohoderhamprop0} gives isomorphisms
\[
 A^p( U_n\backslash \GL_n(\sO_X^{\an}(U)),V)^{\GL_n(\sO_X^{\an}(U))} \cong \Hom_{ \pro(\Ban)}(\hat{\Lambda}^p(\gl_n(\sO_X^{\an}(U))/\fu_n), V)^{U_n}.
\]
Since $U_n$ is connected, $U_n$-invariants correspond to $\fu_n$-invariants, so  
for all trivial $\fu_n$-modules $V$ we have an isomorphism
\[
A^{\bt}(U_n\backslash \GL_n(\sO_X^{\an}(U)), V)^{\GL_n(\sO_X^{\an}(U))}\cong  E^{\bt}_{\cts, \R}(\gl_n(\sO_X^{\an}(U)), \fu_n;V).
\]

 For all  objects $V$ in $\Ch(\pro(\Ban))$, 
 Corollary \ref{cohoderhamprop} thus gives a canonical zigzag
\[
\xymatrix@R=0ex{ 
\CC^{\bt}_{\dif}(G,V) \ar[r] &  \Tot A^0(EG, A^{\bt}(U_n\backslash G , V))^G\\ \ar[ur] E^{\bt}_{\cts, \R}(\gl_n(\sO_X^{\an}(U)), \fu_n;V)}
\]
of  real  quasi-isomorphisms  for all  open polydiscs $U\subset X_{\an}$, where $G=\GL_n(\sO_X^{\an}(U))$. 

We now observe that the functor
\[
V \mapsto  \Tot A^0(EG, A^{\bt}(M , V))^G
\]
is representable in $\Ch(\pro(\Ban)) $ by some object $P(M)$, where $M= U_n\backslash G$. For this, we appeal to \cite[Proposition A.3.1]{descent}: since the functor preserves finite limits,  
 it suffices to show that for any Banach space $V$ and any $ 
\alpha \in A^0(G^i, A^j(M, V) ),
$
the pair $(V, \alpha)$ is dominated by a minimal pair. In other words, we need to find 
 a minimal closed subspace $W$ of $V$ such that $\alpha \in  A^0(G^i, A^j(M, W) )$. To construct this, we can just take the intersection of all such such subspaces; equivalently,  observe that $\alpha$ is given by a map $G^i \by \Lambda^jTM \to V$, and let $W$ be the closure of the span of the image. 

Since $\CC^{\bt}_{\dif}(G,V)=  \HHom_{\Ban}(\ban N\R.BG,V)$ and $N\R.BG$ is cofibrant, the zigzag above gives maps
\[
 \oL\ban N\R.B\GL_n(\sO_X^{\an}(U)) \la P(U_n\backslash \GL_n(\sO_X^{\an}(U))) \to  \hat{E}_{\bt}^{\R}(\gl_n(\sO_X^{\an}(U)), \fu_n)_{\fu_n}.
\]
These  are strict quasi-isomorphisms, $U$ being a polydisc, since they give quasi-isomorphisms on $\HHom(-,V)$ for all pro-Banach spaces $V$, so in particular for fibrant $V$. 
\end{proof}

\begin{remark}\label{smoothDeligneRmk2}
The proof of Proposition \ref{banLiecor} gives a zigzag of strict quasi-isomorphisms
 \[
\oL\ban N\R.B\GL_n(A)\simeq \hat{E}_{\bt}^{\R}(\gl_n(A), \fk)_{\fk}
\]
for any $A$ as in Proposition \ref{cohoderhamprop0}. By Remarks \ref{smoothDeligneRmk}, this applies in particular to the cases $A= \C^{\infty}(\R^d, \R), \C^{\infty}(\R^d, \Cx)$ with $\fk= \fo_n, \fu_n$.
\end{remark}

\subsection{Cyclic homology and Deligne cohomology}

Write $\GL:= \LLim_n \GL_n$, $\gl= \LLim_n \gl_n$ and $\fu:= \LLim_n \fu_n$, with inclusion map $\gl_n \into \gl_{n+1}$ sending the automorphism $g$ of $A^n$ to the automorphism $g\oplus \id$ of $A^{n+1}=A^n \oplus A$. For a  Fr\'echet algebra $A$, write $ N\R.B\GL(A):= \LLim_n N\R.B\GL_n(A)$ whenever $\GL_n(A)$ is a Fr\'echet manifold.

\begin{definition}\label{plusdef}
 If we fix an isomorphism $\alpha \co \Z^{\infty}\oplus \Z^{\infty} \to \Z^{\infty}$, then for any ring $A$, we have a ring homomorphism $\gl(A) \by \gl(A) \to \gl(A)$ given by $(g,h) \mapsto \alpha\circ (g\oplus h) \circ \alpha^{-1}$. We denote this homomorphism simply by $\oplus$. 
\end{definition}

\begin{lemma}\label{banKGLlemma}
For any real   Fr\'echet algebra $A$ for which the $\GL_n(A)$  are Fr\'echet manifolds,  there is a natural  zigzag of strict quasi-isomorphisms
\[
  \oL \ban(N\R.B\GL(A))   \simeq \bigoplus_{n \ge 0} \oL \ban( \Symm^n \fK_{ >0}(A)). 
\]
Under this equivalence, the map $T$ on the left induced by the group homomorphism $\GL(A) \to \GL(A)$ given by $g \mapsto g\oplus g$ is homotopy equivalent to multiplying each summand $\Symm^n( \fK_{ >0}) $ by $2^n$. 
\end{lemma}
\begin{proof}
We know that for any ring $B$, the chain complex $N\R.B\GL(B)= N\CC_{\bt}( B\GL(B), \R) $ is canonically quasi-isomorphic to $ N\CC_{\bt}( B\GL(B)^+, \R)$, by the defining property of the plus construction. For any differential graded Lie algebra $\g$, the loop object has an abelian model:
\[
 \g \by_{ (\g \by \g)}^h 0 \simeq \g[x,dx]\by_{(\ev_0,\ev_1),(\g \by \g)} 0 \xla{\sim} \g.dx,
\]
for $x$ of degree $0$ and $dx\wedge dx=0$. Since $B\GL(B)^+ \simeq K_{>0}(B) $ is a loop space, its Quillen rational  homotopy type is thus represented by the abelian dg Lie algebra  $K_{>0}(B)_{\Q}[1]$, 
giving a quasi-isomorphism between  $ N\CC_{\bt}( B\GL(B)^+, \Q)$ and $ \Symm_{\Q}(K_{>0}(B)_{\Q})$  (as \cite[Theorems I and II]{QRat} naturally extend to nilpotent spaces), so 
\[
 N\R.B\GL(\C^{\infty}(Z,A)) \simeq \Symm_{\R}(\fK_{>0}(A)(Z)) 
\]
for all $Z\in \FrM$. 

There are a multiplication and comultiplication on $ N\R.B\GL(A)$, induced by $\oplus$ and the diagonal on $B\GL(\C^{\infty}(Z,A))$ respectively. The multiplication is homotopy equivalent to the multiplication map on $\Symm(\fK_{>0}(A)(Z))$ regarded as a free commutative dg algebra. The comultiplication map is then homotopy equivalent to the ring homomorphism given on generators by the diagonal $ \fK_{>0}\to \fK_{>0}\oplus \fK_{>0}$. The doubling map $T$ is given by composing the multiplication and comultiplication, so is homotopy equivalent to the ring homomorphism on $ \Symm\fK_{>0}(A)$ which multiplies $\fK_{>0} $ by $2$.
 We now just apply the functor $\oL\ban$.
\end{proof}

\begin{proposition}\label{LieSymmKprop}
There is a canonical zigzag of local strict quasi-isomorphisms
\[
 \hat{E}_{\bt}^{\R}(\gl(\sO_X^{\an}), \fu)_{\fu}\simeq \bigoplus_{n \ge 0} \oL \ban( \Symm^n \fK_{ >0}(\sO_X^{\an})) 
\]
of presheaves of pro-Banach complexes on $\cB(X_{\an})$.
Under this equivalence, the map on the left induced by the ring homomorphism $T\co \gl(\sO_X^{\an}(U)) \to \gl(\sO_X^{\an}(U))$ given by $g \mapsto g\oplus g$ is homotopy equivalent to the map multiplying each summand 
$ \oL\ban \Symm^n\fK_{ >0}(\sO_X^{\an}(U)) $
by $2^n$.
\end{proposition}
\begin{proof}
Taking the colimit over $n$ in Proposition \ref{banLiecor},  we have local strict quasi-isomorphisms
\[
  \oL\ban N\R.B\GL(\sO_X^{\an})\simeq \hat{E}_{\bt}^{\R}(\gl(\sO_X^{\an}), \fu)_{\fu},
\]
which combine with Lemma \ref{banKGLlemma} to give the required quasi-isomorphisms.

To understand the effect of the ring homomorphism $T$, we just observe that this is the map induced on $ \hat{E}_{\bt}^{\R}(\gl(\sO_X^{\an}(U)), \fu)_{\fu}$ by $T^*$ on $A^{\bt}(U_{\infty} \backslash \GL(\sO_X^{\an}(U)),V)$. This in turn is induced by the map on differentiable group cohomology coming from the homomorphism $T$ on $\GL(\sO_X^{\an}(U))$, which by Lemma \ref{banKGLlemma} corresponds to multiplying  
$\oL\ban \Symm^n\fK_{ >0}(\sO_X^{\an}(U))$
by $2^n$.
\end{proof}

\begin{remark}\label{smoothDeligneRmk3}
For any real Fr\'echet algebra $A$ with a compatible system $K_n \subset \GL_n(A)$ of retracts as  in Proposition \ref{cohoderhamprop0}, the proof of Proposition \ref{LieSymmKprop} combines with Remark \ref{smoothDeligneRmk2} to  give a canonical zigzag of  strict quasi-isomorphisms
\[
 \hat{E}_{\bt}^{\R}(\gl(A), \fk)_{\fk}\simeq \bigoplus_{n \ge 0} \oL \ban( \Symm^n \fK_{ >0}(A)), 
\]
where $\fk := \LLim_n \fk_n$. 
\end{remark}

Now, the direct sum $\bigoplus \co \gl \by \gl \to \gl$ of Definition \ref{plusdef} induces a multiplication $\hat{\Lambda}^i\gl(\sO_X^{\an})\hten\hat{\Lambda}^j\gl(\sO_X^{\an})\to \hat{\Lambda}^{i+j}\gl(\sO_X^{\an})$, and as observed in  \cite[(6.5)]{LodayQuillen}, this becomes associative on taking $\gl(\R)$-coinvariants. 
  It is also associative on $\fu$-coinvariants, because $\fu_{2n}$ contains matrices of the form $\left(\begin{smallmatrix} 0 & I \\ -I & 0                                                                                                                    \end{smallmatrix}\right)$,  so $\hat{E}_{\bt}^{\R}(\gl(\sO_X^{\an},\fu))_{\fu}$ becomes a commutative associative dg Fr\'echet algebra.

\begin{definition}
Given a Fr\'echet $k$-algebra $A$ (for $k\in \{\R,\Cx\}$) we define the  continuous cyclic homology complex $\hat{\CC}_{\bt}(A/k)$  by
\[
 \hat{\CC}_n(A/k) := (A^{\hten_k n+1})/(1-t),
\]
where $t(a_0\ten \ldots \ten a_n)= (-1)^n( a_n \ten a_0\ten \ldots,\ten a_{n-1} )$. The differential $b \co\hat{\CC}_n(A/k)\to \hat{\CC}_{n-1}(A/k)$ is given by 
\[
 b(a_0\ten \ldots \ten a_n)=\sum_{i=0}^{n-1}(-1)^i(a_0\ten \ldots\ten a_ia_{i+1}\ten\ldots  \ten a_n + (-1)^n(a_n  a_0\ten \ldots,\ten a_{n-1}).
\]
\end{definition}

\begin{lemma}\label{LieCClemma0}
For a  Fr\'echet $k$-algebra $A$, 
there is a canonical isomorphism
\[
 \hat{E}_{\bt}^{k}(\gl(A))_{\gl(k)} \cong \bigoplus_{n \ge 0} \hat{\Symm}^n_k(\hat{\CC}_{\bt}(A/k)[-1]),
\]
with multiplication given by tensoring symmetric powers. 
\end{lemma}
\begin{proof}
 We adapt the  proof of \cite[Proposition 6.6]{LodayQuillen}. 
There is an isomorphism
\[
 (\hat{\Lambda}^n\gl(A))_{\gl(k)} \cong (k[\Sigma_n]\ten A^{{\hten} n})\ten_{k[\Sigma_n]}\sgn,
\]
where $\sgn$ is the one-dimensional real vector space on which $\Sigma_n$ acts by the signature, and $\Sigma_n$ acts on itself by conjugation. Multiplication is given by $\Sigma_p \by \Sigma_q \to \Sigma_{p+q}$. Decomposing $\Sigma_n$ into conjugacy classes (i.e. cycle types),  it then  follows that  $\hat{E}_{\bt}(\gl(A))_{\gl(\R)}$ is freely generated as a graded-commutative pro-Banach algebra by the image of  $ \bigoplus_n ( W_n\ten A^{\hten n})\ten_{\R[\Sigma_n]}\sgn[-n]$, where $W_n$ denotes the conjugacy class of cyclic permutations. 

For the cyclic group $C_n$, we now just have
\[
 ( W_n\ten A^{\hten n})\ten_{k[\Sigma_n]}\sgn \cong A^{\hten n}\ten_{k[C_n]}\sgn= \hat{\CC}_{n}(A/k)[-1],
\]
and the differential is $b$.
\end{proof}

\begin{lemma}\label{LieCClemma}
 For any complex Fr\'echet algebra $A$, there is  a zigzag of strict quasi-isomorphisms 
\[
 \hat{E}_{\bt}^{\R}(\gl(A), \fu)_{\fu} \simeq \bigoplus_{n \ge 0} \hat{\Symm}^n_{\R}(\hat{\CC}_{\bt}(A/\Cx)[-1]/\bigoplus_{p>0} \R(p)[1-2p]),
\]
where the map $\R(p) \to \hat{\CC}_{2p-2}(A/\Cx)$ is given by $a \mapsto a\ten 1\ten \ldots \ten 1$, for $\R(p)= (2\pi i)^p \R$. 
\end{lemma}
\begin{proof}
Write $\tau \in \Gal(\Cx/\R)$ for the complex conjugation element, and $\bar{A}$ for the Fr\'echet algebra $A$ equipped with the ring homomorphism $\tau\co \Cx\to A$.  Then  $\fu_n\ten_{\R}\Cx \cong \gl_n(\Cx)$ and $\gl_n(A) \ten_{\R}\Cx \cong  \gl_n(A)\by \gl_n(\bar{A})$ under the isomorphisms $u \ten z \mapsto uz$ and  $g\ten z \mapsto (gz, -g^t \bar{z})$, for $u \in \fu_n$, $g \in \gl_n(A)$, $z \in \Cx$.

Under these isomorphisms, the inclusion $\fu_n\ten_{\R}\Cx \into \gl_n(A) \ten_{\R}\Cx$ of Fr\'echet Lie algebras corresponds to the  diagonal map
 $\gl_n(\Cx) \to \gl_n(A)\by \gl_n(\bar{A})$; 
  the action of $\tau$ is given on $\gl_n(\Cx)$ by $h \mapsto   
- h^{\dagger}$, and on $\gl_n(A)\by \gl_n(\bar{A})$ by $(g_1, g_2) \mapsto (-g_2^t, -g_1^t)$.

Observe that 
\[
\hat{E}_{\bt}^{\R}(\gl(A), \fu)_{\fu} = \hat{E}_{\bt}^{\R}(\gl(A))_{\fu}\ten_{E^{\R}_{\bt}(\fu)_{\fu}}\R,
\]
so the isomorphisms above give
\[
 \hat{E}_{\bt}^{\R}(\gl(A), \fu)_{\fu}\ten_{\R}\Cx \cong \hat{E}_{\bt}^{\Cx}(\gl(A) \by \gl(\bar{A}))_{\gl(\Cx)}\ten_{E^{\Cx}_{\bt}(\gl(\Cx))_{\gl(\Cx)}}\Cx.
\]

 The calculations of Lemma \ref{LieCClemma0} give $E^{\Cx}_{\bt}(\gl(\Cx))_{\gl(\Cx)}$ as $\hat{\Symm}_{\Cx}$ applied to  $\hat{\CC}_{\bt}(\Cx/\Cx)[-1]= \bigoplus_{p>0} \Cx[1-2p]$. Now we  need to understand the effect of the operator $\dagger$ on this complex. The action of $\Sigma_n$ on $V^{\ten_{\Cx} n}$ is unitary, so we have $(\sum_{\sigma} \lambda_{\sigma}\sigma)^{\dagger}= \sum_{\sigma} \bar{\lambda}_{\sigma}\sigma^{-1}$. Now, the cyclic permutation $(1,2,\ldots ,(2p-1))$ is conjugate to its inverse by a permutation of sign $(-1)^{p+1}$, so generators of $E^{\R}_{\bt}(\fu)_{\fu} $ consist of elements $a$ in $\bigoplus_{p>0} \Cx[1-2p] $ with $\bar{a}= (-1)^pa$, hence $\bigoplus_{p>0} \R(p)[1-2p]$. 

Similarly, Lemma \ref{LieCClemma0} gives $\hat{E}_{\bt}^{\Cx}(\gl(A) \by \gl(\bar{A}))_{\gl(\Cx)}$ as  $\hat{\Symm}_{\Cx}\hat{\CC}_{\bt}(A \by \bar{A}/\Cx)[-1]$. The description of $\tau$ above also gives  isomorphisms 
\begin{align*}
 \hat{\CC}_{\bt}(A/\Cx)^2 &\xra{(\id, \tau)} \hat{\CC}_{\bt}(A \by \bar{A}/\Cx),\\
\hat{\CC}_{\bt}(A/\Cx) &\xra{(\id+ \tau)} \hat{\CC}_{\bt}(A \by \bar{A}/\Cx)^{\tau},
\end{align*}
and hence an isomorphism between $\hat{E}_{\bt}^{\R}(\gl(A))_{\fu}$ and  $\hat{\Symm}_{\R}\hat{\CC}_{\bt}(A/\Cx)[-1]$.

Putting these calculations together gives
\[
\hat{E}_{\bt}^{\R}(\gl(A), \fu)_{\fu} \cong \bigoplus_{n \ge 0} \hat{\Symm}^n_{\R}(\hat{\CC}_{\bt-1}(A/\Cx)/\bigoplus_{p>0} \R(p)[1-2p]).
\]
\end{proof}


\begin{theorem}\label{mainthm0}
For a complex manifold $X_{\an}$, there exists a zigzag of  strict quasi-isomorphisms 
\[
K_{\Ban, >0}(\sO_X^{\an}(U)) \simeq \hat{\CC}_{\bt}( \sO_X^{\an}(U)/\Cx)[-1]/\bigoplus_{p>0} \R(p)[1-2p],
\]
functorial in polydiscs $U \in \cB(X_{\an})$.
\end{theorem}
\begin{proof}
Combining Lemma \ref{LieCClemma} with Proposition \ref{LieSymmKprop} gives a  zigzag of  strict quasi-isomorphisms
\begin{align*}
 \sL&:=\bigoplus_{m \ge 0}  \oL \ban( \Symm^m_{\R} \fK_{ >0}(\sO_X^{\an}))\\
&\simeq  \bigoplus_{n \ge 0} \hat{\Symm}^n_{\R}(\hat{\CC}_{\bt}(\sO_X^{\an}/\Cx)[-1]/\bigoplus_{p>0} \R(p)[1-2p])=:\sS,
\end{align*}
and we wish to infer that the primitive elements on each side are strictly quasi-isomorphic.

Possibly the most natural way to approach this would be to consider the $\Gamma$-space of summing functors on the monoidal Lie groupoid $(B\GL(\sO_X^{\an}), \oplus)$, and to  extend the strict quasi-isomorphisms above to $\Gamma$-diagrams, noting that all the comparison results extend functorially to finite products of copies of $\GL$. 

Instead, we will deduce existence of the desired strict quasi-isomorphism by considering the effect of the doubling map $T$ from Proposition \ref{LieSymmKprop}. The zigzag above gives a pair of inverses
\[
 (f,g) \in \H^0\oR\HHom_{\Ban,\cB(X_{\an})}(\sL,\sS)\by \H^0\oR\HHom_{\Ban,\cB(X_{\an})}(\sS,\sL). 
\]
Writing $\sL=\bigoplus_m \sL^m$ and $\sS= \bigoplus_n \sS^n$,  
we can compose $f$ with inclusion and projection to obtain a decomposition
\[
 f= f_{11}+ f_{1,\ne 1} + f_{\ne 1,1} + f_{\ne 1, \ne 1} 
\]
with $f_{ij}\in \H^0\oR\HHom_{\Ban,\cB(X_{\an})}(\sL^i,\sS^j)$; we
 wish to show  that $f_{11}$ is invertible.

By Proposition \ref{LieSymmKprop}, we know that the action of $(T-2)$ on $\sL^1$ is objectwise homotopic to $0$, while the action on $\sL^{\ne 1}$ is a quasi-isomorphism.  The description of multiplication in Lemma \ref{LieCClemma} gives the corresponding results for $\sS^1$ and $\sS^{\ne 1}$. Since $f$ commutes with $T-2$, we deduce that the terms $f_{1,\ne 1},f_{\ne 1,1}$ are objectwise homotopic to $0$. Thus $f_{11} + f_{\ne 1, \ne 1}$ must be a strict quasi-isomorphism (being objectwise homotopic to $f$), and so must its direct summand $f_{11}$.   
\end{proof}

\begin{remark}\label{smoothDeligneRmk4}
For any complex  Fr\'echet algebra $A$ for which the spaces $\GL_n(A)$ are Fr\'echet manifolds and the inclusions  $U_n \subset \GL_n(A)$ admit  smooth $U_n$-equivariant deformation retractions, the proof of Theorem \ref{mainthm0} combines with Remark \ref{smoothDeligneRmk3} to  give a zigzag of strict  quasi-isomorphisms
\[
K_{\Ban, >0}(A) \simeq   \hat{\CC}_{\bt}( A/\Cx)[-1]/\bigoplus_{p>0} \R(p)[1-2p].
\]
 By Remarks \ref{smoothDeligneRmk}, this applies in particular to the case $A=  \C^{\infty}(\R^d, \Cx)$.

As observed by Ulrich Bunke, the copies of $\R(p)$  in this context are best thought of as the real completion $(\mathbf{ku})_{\R}$ of the connective $K$-theory spectrum. If we instead started with a real  Fr\'echet algebra $A$ for which the inclusions  $O_n \subset \GL_n(A)$ admitted  smooth $O_n$-equivariant deformation retracts (such as $\C^{\infty}(\R^d, \R)$), similar arguments should give $K_{\Ban}(A)$ in terms of the homotopy fibre of a map from  $(\mathbf{ko})_{\R}$ to real cyclic homology.
\end{remark}

\begin{definition}
 For a  Fr\'echet $k$-algebra $A$, let $\hat{\cB}_{**}(A/k)$ be Connes' double complex \cite[\S 9.8]{W} associated to the cyclic object $A^{\hten_k(\bt+1) }$, and write $\hat{\cB}(A/k)$ for the total complex of $\hat{\cB}_{**}(A/k)$. 

Similarly, if $A$ is a commutative  Fr\'echet $k$-algebra, we write $\Omega^1_{A/k}$ for the module of K\"ahler differentials defined using the topological tensor product $\hten_k$, so  $\Omega^1_{A/k}:= I/\overline{(I\cdot I)}$ for $I= \ker(A\hten_kA \to A)$. This is equipped with a derivation $d\co A \to \Omega^1_{A/k}$ given by $da:=a\ten 1 -1\ten a$. We then set $\Omega^p_{A/k}:= \hat{\Lambda}^p_A(\Omega^1_{A/k})$, and let $\Omega^{\bt}_{A/k}$ be the complex $A \xra{d}\Omega^1_{A/k} \xra{d}\Omega^2_{A/k}\xra{d} \ldots$, with the Hodge filtration $F^p$ defined by brutal truncation.   
\end{definition}

\begin{lemma}\label{CCDRlemma}
If $A$ is a  Fr\'echet $k$-algebra (for $k\in \{\R,\Cx\}$), there exists a canonical zigzag
\[
 \hat{\CC}_{\bt}(A/k) \la \hat{\cB}(A/k) \to\bigoplus_{p\ge 0 }(\Omega^{2p-\bt}_{A/k}/F^{p+1})
\]
in $\Ch(\pro(\Ban))$ to a sum of truncated de Rham complexes. The first map is always a strict quasi-isomorphism, and the second is so whenever $A=\C^{\infty}(M,k)$ for a smooth manifold $M$, or $k=\Cx$ and  $A$ is the ring of complex analytic  functions on an open polydisc. 
\end{lemma}
\begin{proof}[Proof (sketch)] 
Adapting \cite[Lemma 9.6.10 and Proposition 9.8.3]{W}, there is a strict quasi-isomorphism to the  complex $\hat{\CC}_{\bt}(A/k) $ from the total complex $\hat{\cB}(A/k)$. That the quasi-isomorphism is strict follows because the proofs provide explicit contracting homotopies.

The second morphism is now defined on the completed Hochschild complex of $A$ by $a_0\ten \ldots \ten a_n \to \frac{1}{n!}a_0da_1 \wedge \ldots \wedge da_n$. It suffices to show that this is a strict quasi-isomorphism when $A$ is $\C^{\infty}(M,k)$ or the ring of analytic  functions on an open polydisc. In the first case, this follows from the continuous HKR isomorphism of \cite[Theorem 3.3]{pflaumCtsHH}.  For the second case, observe that the Hochschild complex is a model for  $A\hten^{\oL}_{A\hten_k A}A$, where derived tensor products are taken in the sense of \cite{meyerDborn};  we now just construct a Koszul resolution of $A$. If co-ordinates are given by $z_1, \ldots z_d$, then consider the commutative dg Fr\'echet algebra $B$ over $A\hten_k A$ freely generated by $h_1 , \ldots h_d$ in degree $1$, with $dh_i= z_i \ten 1 - 1\ten z_i$. The map $B \to A$ is clearly a strict quasi-isomorphism (using Hadamard's Lemma in the $\C^{\infty}$ case), and $B$ is clearly a projective $A\hten_k A $-module. 
\end{proof}

\begin{remark}\label{smoothDeligneRmk5}
When  $A=  \C^{\infty}(M, \Cx)$ or the ring of complex analytic  functions on an open polydisc, Lemma \ref{CCDRlemma} combines with Remarks \ref{smoothDeligneRmk} and  \ref{smoothDeligneRmk4} to  give a zigzag of strict  quasi-isomorphisms
\[
K_{\Ban, >0}(A) \simeq  \bigoplus_{p>0 }\cone(\R(p) \to \Omega^{-\bt}_{A/\Cx}/F^{p+1})[1-2p]
\]
\end{remark}

\begin{definition}
For a complex manifold $X_{\an}$, define the sheaf $\sA^{\bt}_{X,\R}$ (resp. $\sA^{\bt}_{X,\Cx}$)  to be the de Rham complex of smooth real (resp. complex) forms on $X_{\an}$, regarded as an object of $\Ch(\pro(\Ban),X_{\an})$ via the natural Fr\'echet space structures on spaces of smooth forms. The complex  $\sA^{\bt}_{X,\Cx}$ comes equipped with a bigrading and hence a Hodge filtration $F$.

Following \cite{beilinson}, we write $\R_{X,\cD}(p)$ for the Deligne complex given by the cocone of the map
$\R(p) \to  \Omega^{\bt}_{X/\Cx}/F^p$. We also
write $\sA^{\bt}_{X,\R,\cD}(p)$ for the cocone of the map $\sA^{\bt}_{X,\R}(p) \to \sA^{\bt}_{X,\Cx}/F^p$.
\end{definition}

\begin{theorem}\label{mainthm}
 For a complex manifold $X_{\an}$, there exists a zigzag of local strict quasi-isomorphisms 
\[
 K_{\Ban,>0}(\sO_X^{\an}) \simeq \bigoplus_{p> 0} \R_{X,\cD}(p)^{2p-\bt}
\]
of presheaves of pro-Banach complexes on $X_{\an}$.
\end{theorem}
\begin{proof}
  As we have not established a hypersheafification functor  for pro-Banach presheaves, we first wish to establish that the canonical map $ \R_{X,\cD}(p)\to \sA^{\bt}_{X,\R,\cD}(p)$ is a local strict quasi-isomorphism. This follows by Lemma \ref{localstrictlemma} because integration as in the Poincar\'e lemma provides contracting homotopies on forms on polydiscs. 

 Because $\bigoplus_{p> 0} \sA^{2p-\bt}_{X,\R,\cD}(p)$ is a pro-Banach hypersheaf, the proof of Lemma \ref{localstrictlemma}  gives
\begin{eqnarray*}
 \oR\HHom_{X_{\an}, \Ban}( K_{\Ban,>0}(\sO_X^{\an}), \bigoplus_{p> 0}\sA^{2p-\bt}_{X,\R,\cD}(p) )\\ \simeq \int^h_{U \in \cB(X_{\an})}\oR\HHom_{\Ban}(  K_{\Ban,>0}(\sO_X^{\an}(U)), \bigoplus_{p> 0} \sA^{2p-\bt}_{X,\R,\cD}(p)(U) ),
\end{eqnarray*}
so it suffices to construct a strict zigzag of  quasi-isomorphisms 
\[
 K_{\Ban,>0}(\sO_X^{\an})(U) \simeq \bigoplus_{p>0 } \sA^{2p-\bt}_{X,\R,\cD}(U)
\]
functorial in polydiscs $U$. Philosophically, this is saying that because the category $\cB(X_{\an})$ of open polydiscs is a base for the topology on $X_{\an}$, hypersheaves on $X_{\an}$ correspond to those on $\cB(X_{\an})$.

Lemma \ref{CCDRlemma}  gives a zigzag of local strict quasi-isomorphisms between $\bigoplus_{p> 0} \R_{X,\cD}(p)^{2p-\bt} $ and $ \hat{\CC}_{\bt}( \sO_X^{\an}(U)/\Cx)[-1]/\bigoplus_{p>0} \R(p)[1-2p]$, which combined with Theorem \ref{mainthm0} gives
\[
 K_{\Ban,>0}(\sO_X^{\an}(U)) \simeq \bigoplus_{p> 0} \R_{X,\cD}(p)^{2p-\bt}(U),
\]
as required.
\end{proof}

\begin{definition}\label{KBandef0}
Define the presheaf $K_{\Ban}(\sO_X^{\an})$ of pro-Banach complexes on $X^{\an}$ by 
\[
 K_{\Ban}(\sO_X^{\an}):= \R \oplus K_{\Ban,>0}(\sO_X^{\an}).
\]
\end{definition}

\begin{remark}
 This definition is justified by the observation that $K_0(\sO_X^{\an})$ is locally isomorphic to $\Z$. However, it seems unlikely that $K_{\Ban}(\sO_X^{\an} )$ can be defined in the same way as $K_{\Ban,>0}(\sO_X^{\an})$, because  $K_0(\C^{\infty}(Z,\sO_X^{\an}))$ is in general much larger than $K_0(\sO_X^{\an}) $. 
There is a possibility that for a careful choice of category replacing $\FrM$,  these $K_0$ groups might  have the same image under $\oL\ban$,  because the manifolds  $\GL_n(\sO_X(B))$ all admit $\C^{\infty}$-partitions of unity, being closed subsets of nuclear spaces.
\end{remark}

\begin{corollary}\label{globalKcor}
 For $X$ a smooth proper complex variety, the pro-Banach complex
\[
 K_{\Ban}(X):= \oR\Gamma(X, K_{\Ban}(\sO_X^{\an}))
\]
is strictly quasi-isomorphic to the real Deligne complex
\[
 \bigoplus_{p\ge 0} \oR\Gamma(X_{\an}, \R_{\cD,X}(p))^{2p-\bt} 
\]
equipped with the finest $\R$-linear
topology.
\end{corollary}
\begin{proof}
 From the proof of Theorem  \ref{mainthm}, we have  local strict quasi-isomorphisms $\R_{X,\cD}(p) \to \sA_{\R,X,\cD}(p) $. Since the latter complexes are hypersheaves, a model for $K_{\Ban}(X)$ is given by $\bigoplus_{p\ge 0}\Gamma(X_{\an}, \sA_{\R,\cD,X}(p))^{2p-\bt} $.

The Hodge decomposition of \cite[pp. 94--96]{GriffithsHarris} ensures that inclusion of harmonic forms $\cH^*(X,\R)$ in $\Gamma(X_{\an}, \sA_{X,\R}^{\bt})$ is a strict quasi-isomorphism, giving a strict quasi-isomorphism
\[
 \cocone(\cH^*(X,\R)(p) \to \cH^*(X,\Cx)/F^p)\to \Gamma(X_{\an}, \sA_{\R,\cD,X}(p)).
\]
 Since the first complex is finite-dimensional, it has the finest $\R$-linear topology. It is therefore strictly quasi-isomorphic to $\Gamma(X_{\an}, \sA_{\R,\cD,X}(p)) $ equipped with the finest $\R$-linear topology.
\end{proof}

Beware that Corollary \ref{globalKcor} does not extend to quasi-projective varieties, because their Deligne--Beilinson cohomology is defined using log differential forms.

Observe that for Zariski opens $U \subset X$, the maps $\sO_X(U) \to \sO_X^{\an}(U)$ induce maps $K(\sO_X) \to K_{\Ban}(\sO_X^{\an})$ and hence $K(X) \to K_{\Ban}(X)$. On the level of function complexes with coefficients in $V$, these maps are given by the primitive parts of $\CC^{\bt}_{\dif}(\GL(\sO_X^{\an}(U)),V) \to  \CC^{\bt}(\GL(\sO_X(U)),V)$. From this it follows that the composition of the map  $K(X) \to K_{\Ban}(X)$ with the quasi-isomorphism of Corollary \ref{globalKcor} is just  Beilinson's regulator.

\subsection{Non-connective deloopings}

The pro-Banach complex $K_{\Ban}(X)$ of Corollary \ref{globalKcor} is a rather unsatisfactory hybrid. The presheaf $K_{>0}(\sO_X^{\an})$ is defined by completing $1$-connective $K$-theory, yet $K_{\Ban}(X)$ is non-connective, because we are enforcing descent. One way to resolve this is just to consider the connective part $\tau_{\ge 0} K_{\Ban}(X)$, which applied to Corollary \ref{globalKcor}   gives the absolute Hodge complex, as in \cite{beilinson}. 

An alternative approach would have been to apply the functor $\oL\ban$ to non-connective $K$-theory $\bK$.  As in \cite{ThomasonTrobaugh}, the Bass delooping of $K(X)$ is given by Zariski descent as the cone of $K(X) \to \oR\Gamma_{\Zar}( \bP^1_X, K(\sO))$, and iterating this construction gives non-connective $K$-theory $\bK(X)$.
One could attempt to use this to calculate $\oL\ban \bK$ in the same way we calculated $\oL\ban K_{>0}$. However, this would necessitate an understanding of the topology of $\GL(\sO_X^{\an}[t]) $ and $ \GL(\sO_X^{\an}[t,t^{-1}])$, with suitable symmetric spaces calculating differentiable cohomology.

An alternative and apparently more natural choice of delooping for $K_{\Ban}(X)$ is simply to take the cone of $K_{\Ban}(X) \to K_{\Ban}(\bP^1_X)$, iterating the construction to give a complex $\bK_{\Ban}(X)$.  Since 
\[
\oR\Gamma(\bP^1_X, \R)^{\bt} \simeq \oR\Gamma(X, \R)^{\bt}\oplus \oR\Gamma(X, \R)^{\bt-2}(-1)
\]
as real Hodge complexes, 
our delooping of $K_{\Ban}(X)$ just corresponds under Corollary \ref{globalKcor} to
\[
 \bigoplus_{p\ge -1} \oR\Gamma(X_{\an}, \R_{\cD,X}(p))^{2p+1-\bt},
\]
and we get
\[
 \bK_{\Ban}(X) \simeq \bigoplus_{p\in \Z} \oR\Gamma(X_{\an}, \R_{\cD,X}(p))^{2p-\bt}.
\]
It is reassuring to observe that applying this procedure to  $K_{\Ban,>0}(\sO_X^{\an})$ recovers the same answer, avoiding the somewhat arbitrary choice in  Definition \ref{KBandef0}.

\bibliographystyle{alphanum}
\bibliography{references.bib}
\end{document}

%% file: newcommands.tex
\newcommand\da{\!\downarrow\!}

\newcommand\la{\leftarrow}

\newcommand\id{\mathrm{id}}

\newcommand\ten{\otimes}
\newcommand\hten{\hat{\otimes}}

\newcommand\CC{\mathrm{C}}

\renewcommand\H{\mathrm{H}}
\newcommand\z{\mathrm{Z}}

\newcommand\N{\mathbb{N}}
\newcommand\Z{\mathbb{Z}}
\newcommand\Q{\mathbb{Q}}

\newcommand\R{\mathbb{R}}
\newcommand\Cx{\mathbb{C}}

\newcommand\bK{\mathbb{K}}

\newcommand\bP{\mathbb{P}}

\newcommand\C{\mathcal{C}}

\newcommand\cB{\mathcal{B}}

\newcommand\cD{\mathcal{D}}

\newcommand\cF{\mathcal{F}}

\newcommand\cH{\mathcal{H}}

\newcommand\sA{\mathscr{A}}

\newcommand\sE{\mathscr{E}}
\newcommand\sF{\mathscr{F}}

\newcommand\sH{\mathscr{H}}

\newcommand\sL{\mathscr{L}}

\newcommand\sO{\mathscr{O}}

\newcommand\sS{\mathscr{S}}

\newcommand\sV{\mathscr{V}}

\newcommand\fK{\mathfrak{K}}

\newcommand\g{\mathfrak{g}}

\newcommand\fk{\mathfrak{k}}

\newcommand\fo{\mathfrak{o}}

\newcommand\fu{\mathfrak{u}}

\newcommand\ext{\mathscr{E}\!\mathit{xt}}

\newcommand\Hom{\mathrm{Hom}}

\newcommand\HHom{\underline{\mathrm{Hom}}}

\newcommand\Ext{\mathrm{Ext}}

\newcommand\cone{\mathrm{cone}}
\newcommand\cocone{\mathrm{cocone}}

\newcommand\dif{\mathrm{diff}}

\newcommand\Gal{\mathrm{Gal}}

\newcommand\Ch{\mathrm{Ch}}

\newcommand\Op{\mathrm{Op}}

\newcommand\Spec{\mathrm{Spec}\,}

\newcommand\Lim{\varprojlim}
\newcommand\LLim{\varinjlim}

\newcommand\ho{\mathrm{ho}\!}
\newcommand\into{\hookrightarrow}

\newcommand\xra{\xrightarrow}
\newcommand\xla{\xleftarrow}

\newcommand\bt{\bullet}
\newcommand\by{\times}

\newcommand\Ban{\mathrm{Ban}}
\newcommand\ban{\mathrm{ban}}

\newcommand\Symm{\mathrm{Symm}}

\newcommand\GL{\mathrm{GL}}

\newcommand\gl{\mathfrak{gl}}

\newcommand\Mat{\mathrm{Mat}}

\newcommand\an{\mathrm{an}}

\newcommand\Tot{\mathrm{Tot}\,}

\newcommand\sgn{\mathrm{sgn}}

\newcommand\ev{\mathrm{ev}}

\newcommand\pro{\mathrm{pro}}

\newcommand\pd{\partial}

\newcommand\Sm{\mathrm{Sm}}
\newcommand\FrM{\mathrm{FrM}}

\newcommand\cts{\mathrm{cts}}

\newcommand\Zar{\mathrm{Zar}}

\newcommand\op{\mathrm{opp}}

\newcommand\co{\colon\thinspace}

\newcommand\oR{\mathbf{R}}

\newcommand\oL{\mathbf{L}}

\newcommand\uleft\underleftarrow
\newcommand\uline\underline
\newcommand\uright\underrightarrow

%% file: newthms.tex
\newtheorem{theorem}{Theorem}[section]
\newtheorem{proposition}[theorem]{Proposition}
\newtheorem{corollary}[theorem]{Corollary}

\newtheorem{lemma}[theorem]{Lemma}
\newtheorem*{theorem*}{Theorem}
\newtheorem*{proposition*}{Proposition}
\newtheorem*{corollary*}{Corollary}
\newtheorem*{lemma*}{Lemma}
\newtheorem*{conjecture*}{Conjecture}

\theoremstyle{definition}
\newtheorem{definition}[theorem]{Definition}

\newtheorem*{definition*}{Definition}

\theoremstyle{remark}

\newtheorem{remark}[theorem]{Remark}
\newtheorem{remarks}[theorem]{Remarks}

\newtheorem*{example*}{Example}
\newtheorem*{examples*}{Examples}
\newtheorem*{remark*}{Remark}
\newtheorem*{remarks*}{Remarks}
\newtheorem*{exercise*}{Exercise}
\newtheorem*{property*}{Property}
\newtheorem*{properties*}{Properties}